\documentclass[10pt, a4]{article}

\usepackage{graphics,amsmath,amssymb,amsthm,mathrsfs}
\usepackage[showonlyrefs]{mathtools}
\usepackage{hyperref}
\usepackage{graphicx,color}
\usepackage{xcolor}

\usepackage[showonlyrefs]{mathtools}  %% cites only the labels which are quoted
\mathtoolsset{showonlyrefs=true}

\newtheorem{thm}{Theorem}[section]
\newtheorem{lemma}[thm]{Lemma}
\newtheorem{prop}[thm]{Proposition}
\newtheorem{cor}[thm]{Corollary}
\theoremstyle{definition}
\newtheorem{remark}[thm]{Remark}

\def\R{\mathbb{R}}

\numberwithin{equation}{section}

\begin{document}

\bibliographystyle{amsplain}

\title{Existence of least energy positive solutions to Schr\"{o}dinger systems with mixed competition and cooperation terms: the critical case \\ }

\author{Hugo Tavares $^a$   \quad  Song You $^{b,a}$\\
 \small{ $^a$ CMAFcIO \& Departamento de Matem\'atica, Faculdade de Ci\^encias da Universidade de Lisboa,}\\
         \small{ Campo Grande, 1749-016 Lisboa, Portugal}\\
    \small{$^b$  School of Mathematics and Statistics, Lanzhou University, Lanzhou,}\\
         \small{ Gansu 730000, People's Republic of China}\\
  }
\footnotetext[1]{ E-mails: hrtavares@ciencias.ulisboa.pt (H. Tavares), yous16@lzu.edu.cn (S. You) }
\date{ }
\maketitle
\begin{abstract}
In this paper we investigate the existence of solutions to the following Schr\"{o}dinger system in the critical case
\begin{equation*}
-\Delta u_{i}+\lambda_{i}u_{i}=u_{i}\sum_{j = 1}^{d}\beta_{ij}u_{j}^{2} \text{ in } \Omega, \quad u_i=0 \text{ on } \partial \Omega, \qquad i=1,...,d.
\end{equation*}
Here, $\Omega\subset \mathbb{R}^{4}$ is a smooth bounded domain, $d\geq 2$, $-\lambda_{1}(\Omega)<\lambda_{i}<0$ and  $\beta_{ii}>0$ for every $i$, $\beta_{ij}=\beta_{ji}$ for $i\neq j$, where $\lambda_{1}(\Omega)$ is the first eigenvalue of $-\Delta$ with Dirichlet boundary conditions. Under the assumption that the components are divided into $m$ groups, and that $\beta_{ij}\geq 0$ (cooperation) whenever components $i$ and $j$ belong to the same group, while $\beta_{ij}<0$ or $\beta_{ij}$ is positive and small (competition or weak cooperation) for components $i$ and $j$ belonging to different groups, we establish the existence of nonnegative solutions with $m$ nontrivial components, as well as classification results. Moreover, under additional assumptions on $\beta_{ij}$, we establish existence of least energy positive solutions in the case of mixed cooperation and competition. The proof is done by induction on the number of groups, and requires new estimates comparing energy levels of the system with those of appropriate sub-systems. In the case $\Omega=\R^4$ and $\lambda_1=\ldots=\lambda_m=0$, we present new nonexistence results. This paper can be seen as the counterpart of [Soave-Tavares, J. Differential Equations 261 (2016), 505-537] in the critical case, while extending and improving some results from [Chen-Zou, Arch. Ration. Mech. Anal. 205 (2012), 515--551], [Guo-Luo-Zou, Nonlinearity 31 (2018), 314--339].

\vspace{0.2cm}

\vspace{0.2cm}
\end{abstract}

\section{\bf Introduction}

Consider the following elliptic system with $d\geq 2$ equations:
\begin{equation}\label{S-system}
\begin{cases}
-\Delta u_{i}+\lambda_{i}u_{i}=u_{i}\sum_{j = 1}^{d}\beta_{ij}  u^2_{j}   ~\text{ in } \Omega,\\
u_{i}=0 \text{ on } \partial\Omega,  \quad i=1,...d,
\end{cases}
\end{equation}
where $\Omega=\mathbb{R}^{N}$ and $\lambda_i>0$, or $\Omega\subset \mathbb{R}^{N}$ is a smooth bounded domain and $\lambda_i>-\lambda_1(\Omega)$, $N\leq 4$, $\beta_{ii}>0$ for $i=1,\ldots, d$, $\beta_{ij}=\beta_{ji}\in \mathbb{R}$ for $i\neq j$. Here and in what follows $\lambda_{1}(\Omega)$ denotes the first eigenvalue of $(-\Delta, H_{0}^{1}(\Omega))$. System \eqref{S-system} appears when looking for standing wave solutions $\phi_i(x,t)=e^{\imath \lambda_i t}u_i(x)$ of the corresponding nonlinear Schr\"odinger system
\begin{equation*}
\imath \partial_t \phi_i + \Delta \phi_i + \phi_i \sum_{j=1}^d\beta_{ij}  |\phi_j|^2=0,
\end{equation*}
which has applications in many physical models such
as nonlinear optics (see \cite{Mitchell 1996}) or  Bose-Einstein condensation for multi-species condensates (see \cite{Timmermans 1998}). Physically the $\beta_{ii}$ represent self-interactions within the same component, while the $\beta_{ij}$ $(i\neq j)$ express the strength and the type of interaction between different components $i$ and $j$. When $\beta_{ij}>0$ the interaction is of cooperative type, while $\beta_{ij}<0$ represents competition. The relation $\beta_{ij}=\beta_{ji}$ expresses symmetry in the interaction between different components and provides a variational structure to the problem. Indeed, solutions of \eqref{S-system} correspond to critical points of the energy functional $J:H^1_0(\Omega; \mathbb{R}^d)\to \mathbb{R}$ defined by
\begin{equation*}
J(\mathbf{u}):=\int_{\Omega}\frac{1}{2}\sum_{i=1}^{d}\left(|\nabla u_{i}|^{2}+\lambda_{i}u_{i}^{2}\right)-\frac{1}{4}
\sum_{i,j=1}^{d}\beta_{ij}u^2_{i}u^2_{j}\, dx
\end{equation*}
(where we used the vector notation $\mathbf{u}=(u_{1},\cdots,u_{d})$). In particular, this allows to consider \emph{least energy positive solutions}, which are defined as solutions $\mathbf{u}$ of the system with positive components and achieving the level
\[
\inf\{J(\mathbf{u}):\ J'(\mathbf{u})= 0,\ \mathbf{u}\in H^1_0(\Omega;\R^d),\ u_i> 0\ \forall i\}.
\]
Since the system may admit solutions with trivial components (i.e., $u_i=0$ for some $i$'s), this level may or may not coincide with the \emph{ground state level}:
\begin{equation}\label{eqgroundstatelevel}
\inf\{J(\mathbf{u}):\  J'(\mathbf{u})= 0,\ \mathbf{u}\in H^1_0(\Omega;\R^d), \mathbf{u}\neq \mathbf{0} \}.
\end{equation}
We call a solution \emph{fully nontrivial} when all its components are nontrivial, and \emph{semi-trivial} when some components (but not all) are zero. A solution $\mathbf{u}\neq \mathbf{0}$ is called a \emph{ground state} if it achieves \eqref{eqgroundstatelevel}.

\medbreak

Let us first focus on the subcritical case $N\leq 3$.  Several existence results are available in the literature for the purely cooperative and for the purely competitive cases (respectively $\beta_{ij}>0$ for all $i\neq j$, and $\beta_{ij}<0$ for all $i\neq j$); we refer to the introduction of \cite{Tavares 2016-1} for an overview on the topic and for a complete list of references. In particular, the two equation case ($d=2$) is completely characterized. In this case there is only one interaction parameter, $\beta:=\beta_{12}=\beta_{21}$, and by collecting the results in \cite{Ambrosetti 2007,Zou 2013,Wei 2005-1,Wei 2005,Maia 2006,Mandel 2015,Sirakov 2007} it is known that there exist least energy positive solutions for $\beta\in (-\infty,\underline{\beta})\cup(\overline{\beta},+\infty)$, for some $0<\underline{\beta}\leq \overline{\beta}$; moreover, these solutions are actually ground states for $\beta>\bar \beta$. This holds when $\Omega$ is a bounded domain, or $\Omega=\R^N$ and one works with radially symmetric functions. We remark that there are ranges of parameters for which there are no positive solutions.

For three or more equations ($d\geq 3$) the situation is much richer, since in this case system \eqref{S-system} admits the possible coexistence of cooperation and competition, that is, the existence of pairs $(i_{1}, j_{1})$ and $(i_{2}, j_{2})$ such that $\beta_{i_{1}j_{1}}>0$ and $\beta_{i_{2}j_{2}}<0$. More recently, the existence of least energy positive solutions under simultaneous cooperation and competition has attracted great attention, starting from \cite{Wang 2015,Soave 2015}. In \cite{Soave 2015, Tavares 2016-1}, several existence results are obtained
whenever the $d$ components are divided into groups. These papers are complemented by \cite{BSWang 2016,Wang 2015,Wang 2015-2}, where the $d=3$ component system in a bounded domain is treated, and by \cite{Wang 2019,Wei 2019} in the case $\Omega=\mathbb{R}^N$. On the other hand, the study and classification of ground state solutions is done in \cite{Correia 2016-1,Correia 2016-2}.

\medbreak

All the papers mentioned above deal with the subcritical case. For the critical case $N=4$, when $d=1$ (one equation), system \eqref{S-system} is reduced to the well-known Brezis-Nirenberg problem \cite{Brezis 1983}, where the existence of a positive ground state is shown for $-\lambda_1(\Omega)<\lambda_1<0$. From this perspective, the study of \eqref{S-system} can be seen as a generalization of this classical problem to systems, working with the natural assumption $-\lambda_1(\Omega)<\lambda_i<0$ for every $i$. For the $d=2$  equation case, Chen and Zou \cite{Zou 2012} proved that there exists $0<\beta_{1}<\beta_{2}$ (depending on $\lambda_{i}$ and $\beta_{ii}$) such that \eqref{S-system} has a least energy positive solution if  $\beta_{12}\in (-\infty, \beta_{1}) \cup (\beta_{2}, +\infty)$ (exactly as in the subcritical case). We would also like to point out paper \cite{Zou 2015} where it is shown, for more general powers, that the dimension has a great influence in the existence of least energy positive solutions.

For a system with an arbitrary number of equations in the whole space, Guo, Luo and Zou \cite{Zou 2018} obtained the existence and classification of least energy positive solutions to \eqref{S-system}  under $-\lambda_1(\Omega)<\lambda_{1}=\cdots=\lambda_{d}<0$ in case of a bounded smooth domain of $\R^4$, in the pure cooperative case with some additional technical conditions on the coupling coefficients. In the same paper the case $\lambda_1=\ldots=\lambda_d=0$ and $\Omega=\R^4$ is also treated under the same conditions on the $\beta_{ij}$. The existence and classification of ground states is done in \cite{Yang 2018}.

\medbreak

To conclude the state-of-the art, we would like to mention also a few related problems in the critical case.  In \cite{Tavares 2017}, the first author and A. Pistoia constructed, via a Lyapunov-Schmidt reduction and under appropriate assumptions on the domain $\Omega \subset \mathbb{R}^{4}$, families of positive solutions of \eqref{S-system} in the competitive or weakly cooperative cases with all the components $u_{i}$ blowing up at different points as $\lambda_{i}\rightarrow 0^{-}$. In \cite{Pistoia 2018-2, Wei 2014, Wang 2016}, the authors investigated the existence and multiplicity of fully nontrivial solutions to \eqref{S-system} with $\lambda_{i}=0$ for the critical case in $\mathbb{R}^{N}$. Recently \cite{Pistoia 2018-1,Tavares 2017} are concerned with existence and concentration results for
a Coron-type problem in a bounded domain with one or multiple small holes in the case $\lambda_{i}=0$.
In \cite{Tavares 2019-1}, the first author with D. Cassani and J. Zhang studied the existence of least energy positive solutions in the critical exponential case when $N=2$.

\medbreak

To the best of our knowledge, there are no papers considering \eqref{S-system} with mixed cooperation and competition terms, with $\lambda_{i}\neq 0$ being possibly different and for $\Omega$ a bounded domain.  We are interested in providing conditions on the coefficients that insure the existence of  least energy positive solutions; our main purpose is to extend the results proved in \cite{Tavares 2016-1} (in the subcritical case) to the critical case, thus generalizing to many equations results from \cite{Zou 2012} and improving at the same time results from \cite{Zou 2018}. We refer for the next subsection to the actual statements.

\medbreak

Throughout this text we  always work under the assumptions
\begin{equation}\label{eq:coefficients}
-\lambda_1(\Omega)<\lambda_1,\ldots, \lambda_d<0,\qquad \Omega \text{ a bounded smooth domain of } \mathbb{R}^4,
\end{equation}
where we recall that $\lambda_1(\Omega)$ is the first Dirichlet eigenvalue of the Laplacian, and
\begin{equation}\label{eq:beta_ij}
\beta_{ii}>0 \quad \forall i=1,\ldots, d,\qquad \beta_{ij}=\beta_{ji}\quad \forall i, j=1,\ldots, d,\ i\neq j.
\end{equation} In order to present the main results of this paper, we firstly introduce some notations already used in \cite{Soave 2015, Tavares 2016-1}.
\begin{itemize}
  \item We work with the \emph{coupling matrix} $B:=(\beta_{ij})_{i,j=1,\ldots,d}$.

  \item We endow the space $H^{1}_{0}(\Omega)$ with
  \begin{equation*}
  \langle u,v\rangle_{i}:= \int_{\Omega}\nabla u \cdot\nabla v+\lambda_{i}uv \, dx  \quad \text{ and } \quad \|u\|^{2}_{i}:=\langle u,u\rangle_{i},\qquad \text{ for every } i=1,\ldots, d.
  \end{equation*}
  Observe that these are in fact inner products and norms, respectively, due to assumption \eqref{eq:coefficients}.
  \item Having in mind the idea of organizing the components of a solution to the system into several groups, given an arbitrary $1\leq m \leq d$ we say that a vector $\mathbf{a}=(a_{0},...,a_{m})\in \mathbb{N}^{m+1}$ is an $m$-decomposition of $d$ if
  \begin{equation*}
  0=a_{0}<a_{1}<\cdot \cdot \cdot<a_{m-1}<a_{m}=d.
  \end{equation*}
Given an $m$-decomposition $\mathbf{a}$ of $d$, for $h=1,...,m$ we define
  \begin{align*}
&  I_{h}:=\left\{i\in \{1,...,d\}:a_{h-1}<i\leq a_{h}\right\},
\end{align*}
and
\begin{align*}
&  \mathcal{K}_{1}:=\left\{(i,j)\in I_{h}^{2}  \text{ for some } h=1,...,m, \text{ with } i\neq j\right\},\\
&  \mathcal{K}_{2}:=\left\{(i,j)\in I_{h}\times I_{k} \text{ with } h\neq k\right\}.
  \end{align*}
\end{itemize}
  In this way we  say that $u_i$ and $u_j$ belong to the same group if $(i,j)\in \mathcal{K}_1$ and to a different group if $(i,j)\in \mathcal{K}_2$. As we will see ahead, we will get existence results wherever the interaction between elements of the same group is cooperative, while there is either weak cooperation or competition between elements of different groups.

\medbreak

Next we introduce the main results of this paper. We split them into two subsections: the first one is concerned with existence results for \eqref{S-system} in the case $\Omega$ bounded, the second one with the case $\Omega=\R^4$.

\subsection{ Main results: existence}

Consider the Nehari-type set
\begin{align}\label{Manifold-1}
\mathcal{N}&=
\left\{ \mathbf{u}\in H_{0}^{1}(\Omega;\mathbb{R}^{d}):\sum_{i\in I_{h}}\|u_{i}\|_{i}\neq 0 \text{ and } \sum_{i\in I_{h}}
\partial_{i}J(\mathbf{u})u_{i}=0 \text{ for every } h=1,...,m
\right\}\\
&=\left\{ \mathbf{u}\in H_{0}^{1}(\Omega;\mathbb{R}^{d}): \sum_{i\in I_{h}}\|u_{i}\|_{i}\neq 0 \text{ and }
\sum_{i\in I_{h}}\|u_{i}\|_{i}^2=\sum_{i\in I_{h}}\int_\Omega u_i^2 \sum_{j=1}^d \beta_{ij}u_j^2 \text{ for every } h=1,...,m
\right\},
\end{align}
and the infimum of $J$ on the set $\mathcal{N}$:
\begin{equation}\label{Minimizer-1}
c:=\inf_{\mathbf{u}\in \mathcal{N}}J\left(\mathbf{u}\right).
\end{equation}
The first main result of this paper is the following.

\begin{thm}\label{Theorem-1}
Assume \eqref{eq:coefficients}, \eqref{eq:beta_ij}, let $\mathbf{a}$ be an m-decomposition of $d$ for some $1\leq m\leq d$, and assume that
$\beta_{ij}\geq 0$  $\forall (i,j)\in \mathcal{K}_{1}$.
There exists $\Lambda>0$ such that, if
\begin{equation*}
-\infty<\beta_{ij}< \Lambda \quad\forall (i,j)\in \mathcal{K}_{2},
\end{equation*}
then $c$ is attained by a nonnegative $\mathbf{u}\in \mathcal{N}$. Moreover, any minimizer is a  solution of \eqref{S-system}.
\end{thm}

\begin{remark}\label{2}
We mention that $\Lambda$ is dependent on $\beta_{ij}\geq 0$ for $(i,j)\in \mathcal{K}_{1}$, $\beta_{ii}, \lambda_{i}, i=1,\ldots, d$ (see \eqref{Constant-8} ahead for the explicit expression).
\end{remark}

\begin{remark}\label{2-2}
In \cite{Tavares 2016-1} the authors proved that the above theorem holds true in the subcritical case $N\leq 3$ in a bounded domain $\Omega \subset \mathbb{R}^{N}$, where however the compactness of $H^{1}_{0}(\Omega)\hookrightarrow L^{4}(\Omega)$  plays a key role. This loss of compactness makes  problem \eqref{S-system} very complicated and substantially different, requiring new ideas. In \cite[Theorem 1.2]{Tavares 2016-1} the constant $\Lambda$ is only dependent on $\beta_{ii}, \lambda_{i}$;  affected by the presence of a critical exponent in \eqref{S-system}, the constant $\Lambda$ in our statement is also dependent on $\beta_{ij}$ for $(i,j)\in \mathcal{K}_{1}$.
\end{remark}

Observe that, when we are dealing with one single group ($m=1$, $\mathbf{a}=(0,d)$), the level $c$ reduces to
\begin{equation*}
\widetilde{c}:=\inf_{\mathcal{\widetilde N}}J\left(\mathbf{u}\right),\quad \text{ where }  \quad\mathcal{\widetilde N}:= \Big\{\mathbf{u}\in H_{0}^{1}(\Omega;\mathbb{R}^{d}):\mathbf{u}\neq \mathbf{0} \text{ and }
\langle\nabla J(\mathbf{u}),\mathbf{u}\rangle=0 \Big\}.
\end{equation*}
As a consequence of Theorem \ref{Theorem-1} we can thus obtain the existence of \emph{ground state} solutions (recall the definition in \eqref{eqgroundstatelevel}) for the purely cooperative case, as well as its classification in case $\lambda_1=\ldots=\lambda_d$. Following \cite{Correia 2016-1,Tavares 2016}, define $f: \mathbb{R}^{d}\mapsto \mathbb{R}$ by
\begin{equation}\label{f-define}
f(x_{1}, \cdots, x_{d})=\sum_{i,j=1 }^{d}\beta_{ij}x_{i}^{2}x_{j}^{2},\quad \text{ let } \quad f_{max}:=\max_{|X|=1}f(X),
\end{equation}
and denote by $\mathcal{X}$ the set of elements in the unit sphere of $\R^d$ achieving $f_{max}$.

\begin{cor}\label{Classification}
Assume \eqref{eq:coefficients},\eqref{eq:beta_ij}, and that
\[
\beta_{ij}\geq 0 \quad \text{ for } i,j=1,\ldots, d.
\]
Then $\widetilde{c}$ is achieved by a nonnegative $\mathbf{u}\in \mathcal{\widetilde N}$, which is a ground state solution of \eqref{S-system}. Moreover, when $\lambda:=\lambda_1=\ldots=\lambda_d$, $\mathbf{u}$ is a ground state of \eqref{S-system} if and only if $\mathbf{u}=X_{0}U$, where $X_{0} \in \mathcal{X}$ and $U$ is a positive ground state solution of
\begin{equation}\label{Class-2-4}
-\Delta v +\lambda v =f_{max}v^{3} ~\text{ in } \Omega.
\end{equation}
\end{cor}

We mention that $\widetilde{c}$ is achieved by Theorem \ref{Theorem-1}. The classification result is shown by applying the strategy of \cite[Theorem 2.1]{Tavares 2016} (see also \cite[Theorem 1]{Correia 2016-1}). Clearly, whenever each component of $X_0$ is nonzero, by the maximum principle and the invariance of $J$ and $ \mathcal{\widetilde N}$ under the transformation $\mathbf{u}\mapsto (|u_1|,\ldots, |u_d|)$, the ground state level coincides with the least energy positive level.

\begin{remark}\label{1}
We point out that in \cite[Theorem 1.2]{Zou 2012}, for the two equations case $d=2$ and $\lambda:=\lambda_1=\lambda_2$, it is proved that least energy positive solutions have the form $(u,v)=(\sqrt{k}\omega, \sqrt{l}\omega)$ for $0<\beta_{12}<\min\{\beta_{11}, \beta_{22}\}$ or $\beta_{12}>\max\{\beta_{11}, \beta_{22}\}$, where $\omega$ is a positive ground state solution of \eqref{Class-2-4} with $f_{max}=1$. Moreover, a pair $(\sqrt{k}\omega,\sqrt{l}\omega)$ is a positive least energy solution if $k,l$ solve a certain linear system (see \cite[Theorem 1.1]{Zou 2012}). For a system with an arbitrary number of equations, $\lambda:=\lambda_1=\ldots=\lambda_d$ and $\beta_{ij}\geq 0$, \cite[Theorem 1.1]{Zou 2018} presented the existence of least energy positive solutions of the form $u_{i}=\sqrt{c_{i}}\omega, i=1,\ldots,d$ for $B$ an invertible matrix such that the sum of each column of $B^{-1}$ is greater than $0$, obtaining also in \cite[Theorem 1.2]{Zou 2018} that all minimizers are of this form if moreover $B$ is positive or negative definite. By a direct computation we deduce that \cite[Theorems 1.1 \& 1.2]{Zou 2012} and \cite[Theorem 1.1 \& 1.2]{Zou 2018} are a special case of Corollary \ref{Classification}.
\end{remark}

Based on Theorem \ref{Theorem-1}, we obtain least energy positive solutions of \eqref{S-system} by following ideas from \cite{Tavares 2016-1}. Firstly, using Theorem \ref{Theorem-1} in the particular situation $m=d$ (so that $\mathbf{a}=(0,1,2,\ldots, d)$ necessarily), we obtain the existence of least energy positive solution of \eqref{S-system} under competition and/or weak cooperation.
\begin{cor}\label{Theorem-1-1}
Assume \eqref{eq:coefficients} and \eqref{eq:beta_ij}.
There exists $\Lambda>0$, depending only on $\beta_{ii}, \lambda_{i} ~(i=1,\ldots, d)$ such that, if
\begin{equation*}
-\infty<\beta_{ij}<\Lambda \quad\forall i\neq j,
\end{equation*}
then \eqref{S-system} has a least energy positive solution.
\end{cor}

The following two results allow strong cooperation between elements which belong to the same group. They correspond to Theorem 1.4 and Theorem 1.5 in \cite{Tavares 2016-1} (which dealt with the subcritical case $N\geq 3$)

\begin{thm}\label{Theorem-2}
Assume \eqref{eq:coefficients}, \eqref{eq:beta_ij}, and let $\mathbf{a}$ be an m-decomposition of $d$ for some $1\leq m\leq d$. Let $\Lambda$ be the constant defined in Theorem \ref{Theorem-1}. If
\begin{enumerate}
  \item  $\lambda_{i}=\lambda_{h}$ for every $i\in I_{h}, h=1,\ldots, m$;
  \item $\beta_{ij}=\beta_{h}>\max\{\beta_{ii}: i\in I_{h}\}$ for every $(i,j)\in I_{h}^{2}$ with $i\neq j, h=1,\ldots, m$.
  \item $\beta_{ij}=b<\Lambda$ for every $(i,j)\in \mathcal{K}_{2}$;
\end{enumerate}
then system \eqref{S-system} has a least energy positive solution.
\end{thm}

\begin{thm}\label{Theorem-3}
Assume \eqref{eq:coefficients}, \eqref{eq:beta_ij}, and let $\mathbf{a}$ be an m-decomposition of $d$ for some $1\leq m\leq d$. Let $\Lambda$ be the constant defined in Theorem \ref{Theorem-1} and fix $\alpha>1$. If
\begin{enumerate}
  \item $\lambda_{i}=\lambda_{h}$ for every $i\in I_{h}, h=1,\ldots, m$;
  \item $\beta_{ij}=\beta_{h}>\frac{\alpha}{\alpha-1}\max_{i\in I_h}\{\beta_{ii}\}$ for every $(i,j)\in I_{h}^{2}$ with $i\neq j, h=1,\ldots, m$;
  \item    $ |\beta_{ij}|\leq \frac{\Lambda}{\alpha d^{2}}$ for every $(i,j)\in \mathcal{K}_{2}$;
\end{enumerate}
then system \eqref{S-system} has a least energy positive solution.
\end{thm}

In the next subsection, we describe our main results on the case $\Omega=\R^4$ for $\lambda_1=\ldots=\lambda_k=0$.

%%%%
%
%%%

\subsection{ Main results: the limiting system case}

In \cite{Zou 2012}, for the two equation case ($d=2$), in order to prove the existence of fully nontrivial solutions an important role is played by the limiting equation
\begin{equation}\label{eq:single_critical}
-\Delta u=u^3 \qquad \text{ in } \R^4,
\end{equation}
whose positive solutions in $\mathcal{D}^{1,2}(\R^4)$ are given by
\begin{equation}\label{Class-2}
U_{\epsilon,y}(x):=\frac{2\sqrt{2}\epsilon}{\epsilon^{2}+|x-y|^{2}},\qquad \epsilon>0,\ y\in \R^4.
\end{equation}
In our situation (arbitrary number of equations and mixed cooperation/competition parameters), the role of  \eqref{eq:single_critical} is replaced by the role of the following sub-systems
\begin{equation}\label{sub-system}
\begin{cases}
-\Delta v_{i}=\sum_{j \in I_{h}}\beta_{ij}v_{j}^{2}v_{i}  ~\text{ in } \mathbb{R}^{4},\\
v_{i}\in \mathcal{D}^{1,2}(\mathbb{R}^{4})  \quad \forall i\in I_{h}.
\end{cases}
\end{equation}
Existence and classification results for ground states of this system were shown in \cite{Yang 2018}. Here we complement such result by presenting new characterizations of the ground state level in the purely cooperative case, which will be of key importance in the nonexistence result we present at the end of the introduction.

\subsubsection{Existence, classification and characterization of ground state solutions for sub-systems}

Set $\mathbb{D}_h:=(\mathcal{D}^{1,2}(\mathbb{R}^{4}))^{a_{h}-a_{h-1}}$ with the norm $\|\mathbf{u}\|_{\mathbb{D}_h}:=\left(\sum_{i\in I_{h}}\int_{\mathbb{R}^{4}}|\nabla u_{i}|^{2}\right)^{\frac{1}{2}}$. Take the energy functional
\begin{equation*}
E_{h}(\mathbf{v}):=\int_{\mathbb{R}^{4}}\frac{1}{2}\sum_{i\in I_{h}}|\nabla v_{i}|^{2}-\frac{1}{4}\sum_{(i,j)\in I_{h}^{2}}
\beta_{ij}v_{j}^{2}v_{i}^{2}\, dx,
\end{equation*}
and well as the level
\begin{equation}\label{eq:Ground State-2}
l_{h}:= \inf_{\mathcal{M}_{h}}E_{h},\quad \text{ with } \quad  \mathcal{M}_{h}:=\Big\{\mathbf{v}\in \mathbb{D}_h:\mathbf{v}\neq \mathbf{0} \text{ and }
\langle\nabla E_{h}(\mathbf{v}),\mathbf{v}\rangle=0 \Big\}.
\end{equation}
Assume $\beta_{ij}\geq 0$ for every $(i,j)\in I_h^2$ with $\beta_{ii}>0$. It is standard to prove that $l_{h}>0$, and that
\begin{align}\label{Vector Sobolev Inequality-1}
 l_{h}  = \inf_{\mathbf{v}\in\mathbb{D}_h\backslash \{\mathbf{0}\}}\max_{t>0}E_{h}(t\mathbf{v})
    = \inf_{\mathbf{v}\in\mathbb{D}_h\backslash \{\mathbf{0}\}}\frac{1}{4}\frac{\left(\int_{\mathbb{R}^{4}}\sum_{i\in I_{h}}|\nabla v_{i}|^{2}\right)^{2}}{\int_{\mathbb{R}^{4}}\sum_{(i,j)\in I_{h}^{2}}
\beta_{ij}v_{j}^{2}v_{i}^{2}}.
\end{align}
Therefore we get the following vector Sobolev inequality
\begin{equation}\label{Vector Sobolev Inequality}
4l_{h}\int_{\mathbb{R}^{4}}\sum_{(i,j)\in I_{h}^{2}}\beta_{ij}v_{j}^{2}v_{i}^{2}\leq \left(\int_{\mathbb{R}^{4}}\sum_{i\in I_{h}}|\nabla v_{i}|^{2}\right)^{2},~\forall \mathbf{v}\in \mathbb{D}_h,
\end{equation}
which will play an important role in the study of the system \eqref{S-system}. In order to state an alternative characterization of ground states we also introduce, for $h=1,...,m$,
\begin{equation*}
\widetilde{E}_{h}(\mathbf{v}):=\frac{1}{4}\int_{\mathbb{R}^{4}}\sum_{i\in I_{h}}|\nabla v_{i}|^{2}=\frac{1}{4}\|\mathbf{v}\|_{\mathbb{D}_h}^{2},
\end{equation*}
and
\begin{equation}\label{eq:tilde_l_h}
\widetilde{l}_{h}:= \inf_{\widetilde{\mathcal{M}}_{h}}\widetilde{E}_{h}, \quad \text{ with } \quad \widetilde{\mathcal{M}}_{h}:=\Big\{\mathbf{v}:\mathbf{v}\neq \mathbf{0} \text{ and } \|\mathbf{v}\|_{\mathbb{D}_h}^{2}\leq \int_{\mathbb{R}^{4}}\sum_{(i,j)\in I_{h}^{2}}
\beta_{ij}v_{j}^{2}v_{i}^{2}
 \Big\}.
\end{equation}
Clearly, $\widetilde l_h \leq l_h$.

Finally, consider $f_h: \mathbb{R}^{|I_{h}|}\mapsto \mathbb{R}$ defined by
\begin{equation}\label{f-define}
f_h(x_{1}, \cdots, x_{|I_{h}|})=\sum_{i,j=1 }^{|I_{h}|}\beta_{ij}x_{i}^{2}x_{j}^{2},
\end{equation}
and denote by $\mathcal{X}_h$ the set of solutions to the maximization problem
\begin{equation}\label{Class-1}
f_h(X_{0})=f^h_{max}:=\max_{|X|=1}f_h(X),\quad |X_{0}|=1.
\end{equation}
\begin{thm}\label{limit system-6-1}
Assume  \eqref{eq:coefficients} and that $\beta_{ij}\geq 0$  $\forall (i,j)\in I_{h}^{2}$, $\beta_{ii}>0$. Then $l_{h}=\widetilde l_h$, and any minimizer for $\widetilde{l}_{h}$ is a minimizer for $l_{h}$. This level is  attained by a nonnegative $\mathbf{V}_{h}$, a solution of \eqref{sub-system}. Moreover, any of such minimizers has the form $\mathbf{V}_{h}=X_{0}(f^h_{max})^{-\frac{1}{2}}U_{\epsilon,y}$, where $X_{0} \in \mathcal{X}_h$, $y\in \R^4$, $\epsilon>0$.

\end{thm}
\begin{remark}\label{7-3}
The existence and classification of ground states has been established in \cite{Yang 2018} for more general assumptions on the coefficients $\beta_{ij}$ and on the exponents. In the cooperative case we complement these results by providing a characterization in terms of a Nehari manifold, which is crucial in the proof of Theorem \ref{Theorem-1}. Moreover, we provide the characterization $l_h=\widetilde l_h$, which is crucial to present the nonexistence results for the limiting system (namely in the proof of the forthcoming Theorem \ref{limt system-6}).
\end{remark}

\begin{remark}\label{7}
Every $\mathbf{V}_{h}$ is a least energy positive solution of \eqref{sub-system} when each component of $X_{0}$ is not zero. We point out that both $X_{0}$ and $f^h_{max}$ are only dependent on $\beta_{ij}$ for $(i,j)\in I_{h}^{2}$.
\end{remark}

In the next subsection we will study the nonexistence of least energy solutions for $\Omega=\R^4$, $\lambda_1=\ldots=\lambda_d=0$. We remark that such result is independent of Theorem \ref{Theorem-1}, as it plays no role in its proof.

\subsubsection{A nonexistence result for $\Omega=\R^4$}

Observe that if $\Omega=\mathbb{R}^{4}$ and $\mathbf{u}$ is any a solution of \eqref{S-system}, then by the Pohozaev Identity and $\langle \nabla J(\mathbf{u}),\mathbf{u}\rangle=0$, it is easy to see that $\int_{\mathbb{R}^{4}}\sum_{i=1}^{d}\lambda_{i}u_{i}^{2}=0$. This yields that $\mathbf{u}\equiv \mathbf{0}$ if $\lambda_{1},\ldots,\lambda_{d}$ have the same sign. A natural question is what happens in the limiting case $\lambda_1,\ldots, \lambda_d=0$. Are there least energy positive solutions, or at least nonnegative solutions with nontrivial grouping? In this section we state a nonexistence result under assumptions of cooperation between elements withing the same group, competition between elements of different groups.

Consider the limiting system:
\begin{equation}\label{limit-system}
\begin{cases}
-\Delta u_{i}=\sum_{j=1}^{d}\beta_{ij}u_{j}^{2}u_{i}  ~\text{ in } \mathbb{R}^{4},\\
u_{i}\in \mathcal{D}^{1,2}(\mathbb{R}^{4})  \quad \forall i=1,...,d.
\end{cases}
\end{equation}
Define the associated energy
\begin{equation*}
E\left(\mathbf{u}\right):=\sum_{h=1}^{m}\int_{\mathbb{R}^{4}}\frac{1}{2}\sum_{i\in I_{h}}|\nabla u_{i}|^{2}-\frac{1}{4}\sum_{h,k=1}^{m}\int_{\mathbb{R}^{4}}\sum_{(i,j)\in I_{h}\times I_{k}}\beta_{ij}u_{j}^{2}u_{i}^{2}\,
\end{equation*}
and the level
\begin{equation}\label{Minimizer-2}
l:=\inf_{\mathcal{M}}E\left(\mathbf{u}\right),
\end{equation}
where
\begin{align}
\mathcal{M}:=
\Bigg\{ \mathbf{u}\in \mathcal{D}^{1,2}(\mathbb{R}^{4};\mathbb{R}^{d}): \sum_{i\in I_{h}}\|u_{i}\|_{i}\neq 0 \text{ and } \sum_{i\in I_{h}}
\partial_{i}E(\mathbf{u})u_{i}=0, \text{ for every } h=1,...,m
\Bigg\}.
\end{align}
Our last main result is the following.

\begin{thm}\label{limt system-6}
Assume \eqref{eq:beta_ij} and let $\mathbf{a}$ be an m-decomposition of $d$ for some $2\leq m\leq d$. If
\begin{itemize}
  \item $\beta_{ij}\geq 0$,  $\forall (i,j)\in \mathcal{K}_{1}$;
  \item $\beta_{ij}\leq 0$,  $\forall (i,j)\in \mathcal{K}_{2}$, and there exists $h_{1}\neq h_{2}$ such that $\beta_{ij}< 0$ for every $(i,j)\in I_{h_{1}}\times I_{h_{2}}$;
\end{itemize}
then $l$ is not achieved and $l=\sum_{h=1}^{m}l_{h}$.
\end{thm}

Clearly, as a byproduct of this result, under the previous assumptions all ground state solutions of \eqref{limit-system} have some components that vanish, and there are no least energy positive solutions.
\begin{remark}\label{1.3-2}
We observe that the first authors and N. Soave  in \cite{Tavares 2016-1} proved $l$ is not achieved for the subcritical case $N\leq 3$. Here, based on Theorem \ref{limit system-6-1}, we show that the above result holds also for critical case.
\end{remark}

\begin{remark}
This result plays no role in the proof of Theorem \ref{Theorem-1}. Actually, the later does not depend on any existence result for the limiting system \eqref{limit-system}, only on results for the sub-systems \eqref{sub-system}.
\end{remark}
Applying Theorem \ref{limt system-6} in the particular case $m=d$, we obtain the following.

\begin{cor}\label{1-4}
If
\begin{itemize}
  \item $\beta_{ii}> 0$,  for every $i=1,\ldots, d$;
  \item $\beta_{ij}\leq 0$ for every $i\neq j$, $i,j=1,\ldots, d$;
  \item there exists $i_{1}\neq j_{1}$ such that $\beta_{i_{1}j_{1}}< 0$;
\end{itemize}
then $l$ is not achieved and $l=\sum_{h=1}^{d}l_{h}$.
\end{cor}

Up to our knowledge, \cite{Zou 2018} is the only reference considering nonexistence results for problem \eqref{Minimizer-2} in the critical case with mixed coefficients. Our Theorem \ref{limt system-6} improves \cite[Theorems 1.4 \& 1.5]{Zou 2018}, which deal with a situation with $m=2$ (two groups) and an arbitrary number of equations, together with some technical conditions for the coefficient matrix $B$.

\subsection{Structure of the paper}
In Section \ref{sec2} we prove Theorem \ref{limit system-6-1} and Theorem \ref{limt system-6}. Section \ref{sec3} is devoted to some auxiliary results which will ultimately lead to Theorem \ref{Theorem-1}. In Subsection \ref{subsec3.1} and Subsection \ref{subsec3.2} we present some new energy estimates, see Lemma \ref{Pre-1}, Theorem \ref{Energy Estimates} and Theorem \ref{Energy Estimation-1,2,m-1}, which are important to prove Theorem \ref{Theorem-1}. We construct a Palais-Smale sequence at level $c$ in Subsection \ref{subsec3.3}. Section \ref{sec4} is devoted to the proofs of the main theorems on a bounded domain. In Subsection \ref{subsec4.1} we show that Theorem \ref{Theorem-1} holds for the case of one group. Subsection \ref{subsec4.2} is then devoted to the proof of Theorem \ref{Theorem-1} in the general case, as well as to the proof of its corollaries.

Theorem \ref{Energy Estimation-1,2,m-1} compares the energy between all possible subsystems. Its proof is inspired by
\cite[Lemma 5.1]{Zou 2012}; however, because of the presence of multi-components, the method in \cite{Zou 2012} cannot be used here directly, and we need some crucial modifications for our proof. The fact that we are dealing with many components does not allow to perform explicit computations as in the two equation case; this is for instance the case in order to show positive definiteness of certain matrices, and so we rely on the notion of strictly diagonally dominant matrices as in \cite{Soave 2015,Tavares 2016-1}, see Lemma \ref{Con-4-1} and Proposition \ref{PS Sequence}. On the other hand, whenever projecting in the Nehari manifold one cannot in general obtain the explicit expression of the coefficients, and we rely on qualitative estimates instead (check for instance Lemma \ref{t-positive} ahead).

 The loss of compactness due to the appearance of the Sobolev critical exponent makes it difficult to acquire the existence of fully nontrivial solutions to \eqref{S-system}. For $d=2$, Chen and Zou \cite{Zou 2012} obtained their existence by comparing the least energy level to \eqref{S-system} with that of limit system ($\Omega=\mathbb{R}^{4}$ and $\lambda_{i}=0$) and scalar equations. They mainly use the equation \eqref{eq:single_critical} to estimate the energy level. However, for the mixed case with multi-components, \eqref{S-system} becomes very complicated, and we need some new ideas for our proof.
 We separate the components into $m$ groups as in \cite{Soave 2015,Tavares 2016-1}, and require that the interaction between components of the same group is cooperative, while the interaction between components belonging to different groups is competitive or weakly cooperative. We call each group as a sub-system and investigate the ground state level of the sub-system. Since the system \eqref{S-system} involves multi-components, we need to establish  new estimates (see Theorem \ref{Energy Estimates} and Theorem \ref{Energy Estimation-1,2,m-1}). Then we can compare energy levels of the system with those of appropriate sub-systems and sub-groups, and obtain the existence of nonnegative solutions with $m$ nontrivial components by induction on the number of groups. Moreover, under additional assumptions on $\beta_{ij}$, we acquire existence of least energy positive solutions of \eqref{S-system}. We stress that our method does not require a comparison between the level $c$ and the ground state level of the limiting system \eqref{limit-system}.

\subsection{Further notations}\label{subsec1.4}

\begin{itemize}
  \item The $L^{p}(\Omega)$ norms will be denoted by $|\cdot|_{p}$, $1\leq p\leq \infty$.
  \item Let
  \begin{equation}\label{eq:def_of_S}
  S:= \inf_{i=1,\ldots,d}\inf_{u\in H^{1}_{0}(\Omega)\setminus \{0\}}\frac{\|u\|^{2}_{i}}{|u|^{2}_{4}}.
  \end{equation}
  By the Sobolev embedding $H^1_0(\Omega)\hookrightarrow L^4(\Omega)$ and since $\lambda_i\in (-\lambda_1(\Omega),0)$, we have $S>0$. Moreover,
  \begin{equation}\label{Constant-1}
    S|u|^{2}_{4}\leq \|u\|^{2}_{i}\leq \int_{\Omega}|\nabla u|^2, ~\forall u \in H^{1}_{0}(\Omega).
  \end{equation}
  \item For $1\leq k\leq d$ and $\mathbf{u}=(u_{1},\cdots, u_{k})$, denote $|\nabla \mathbf{u}|^{2}:=\sum_{i=1}^{k}|\nabla u_{i}|^{2}$ and $|\mathbf{u}|:=(|u_{1}|,\cdots, |u_{k}|)$.
  \end{itemize}
  Take $\mathbf{a}=(a_1,\ldots, a_m)$ an $m$-decomposition of $d$, for some integer $m\in [1,d]$.
  \begin{itemize}
  \item Given $\mathbf{u}=(u_1,\ldots, u_d)\in H^{1}_{0}(\Omega; \mathbb{R}^{d})$ we set
  \begin{equation*}
  \mathbf{u}_{h}:=\left(u_{a_{h-1}+1},\ldots,u_{a_{h}}\right)\qquad \text{ for $h=1,\ldots,m$}.
  \end{equation*}
This way, we have $\mathbf{u}=(\mathbf{u}_1,\ldots, \mathbf{u}_m)$ and $H^1_0(\Omega;\R^d)= \prod_{h=2}^{m} H^1_0(\Omega)^{a_{h}-a_{h-1}}$.  Each space $(H^{1}_{0}(\Omega))^{a_{h}-a_{h-1}}$ is naturally endowed with the following scalar product and norm
  \begin{equation*}
  \langle\mathbf{u}^{1},\mathbf{u}^{2}\rangle_{h}:=\sum_{i\in I_{h}}\langle u^{1}_{i},u^{2}_{i}\rangle_{i} \text{ and }
  \|\mathbf{u}\|^{2}_{h}:=\langle\mathbf{u},\mathbf{u}\rangle_{h}.
  \end{equation*}
Sometimes we will use the notation $\mathbb{H}:= H^{1}_{0}(\Omega; \R^d)$, with norm $\|\mathbf{u}\|^{2}=\sum_{h=1}^{m}\|\mathbf{u}_{h}\|_{h}^{2}$.

  \item Let $\Gamma\subseteq \{1,\ldots,m\}$ and set $d_\Gamma:=|\cup_{k\in\Gamma}I_k|$. For $\mathbf{u}\in H^{1}_{0}(\Omega, \mathbb{R}^{d_\Gamma})$ define the $|\Gamma|\times |\Gamma|$ matrix
    \begin{equation*}
  M_{B}^\Gamma(\mathbf{u}):= \left( \sum_{(i,j)\in I_{h}\times I_{k} }\int_{\Omega}\beta_{ij}u_{i}^{2}u_{j}^{2}\right)_{h, k\in \Gamma}.
  \end{equation*}
Consider
  \begin{align}\label{Diagonally Dominant}
  \mathcal{E}_\Gamma&:= \left\{\mathbf{u}\in H^{1}_{0}(\Omega, \mathbb{R}^{d_\Gamma}): M_{B}^\Gamma(\mathbf{u}) \text{ is strictly diagonally dominant } \right\}\nonumber\\
  		&=\left\{\mathbf{u}\in H^{1}_{0}(\Omega, \mathbb{R}^{d_\Gamma}): \left|\sum_{(i,j)\in I_{h}^{2}}\int_\Omega \beta_{ij} u_i^2u_j^2 \right|  > \sum_{k\in \Gamma,k\neq h} \left|\sum_{(i,j)\in I_{h}\times I_{k} }\int_\Omega \beta_{ij} u_i^2u_j^2 \right|, \text{ for  } h\in\Gamma  \right\}.
  \end{align}
 For simplicity, denote $\mathcal{E}:=\mathcal{E}_\Gamma$ and $M_{B}(\mathbf{u}):=M_{B}^\Gamma(\mathbf{u})$ when $\Gamma=\{1,\ldots,m\}$.
  \end{itemize}
Recall that a square matrix that is strictly diagonally dominant and has positive diagonal terms is positive definite. Therefore, $M_{B}^\Gamma(\mathbf{u})$ is positive definite if $\mathbf{u}\in\mathcal{E}_\Gamma$ and $\beta_{ij}\geq 0$ for every $(i,j)\in \mathcal{K}_1$.

\section{\bf The limit system}\label{sec2}
In this section, we prove Theorem \ref{limit system-6-1} and Theorem \ref{limt system-6}.  Take $\mathbf{a}=(a_1,\ldots, a_m)$ an $m$-decomposition of $d$, for some integer $m\in [1,d]$, and fix $h\in \{1,\ldots, m\}$.

\begin{lemma}\label{2.2-2}
Assume $\beta_{ij}\geq 0$  $\forall (i,j)\in I_{h}^{2}$ and $\beta_{ii}>0$. Then ${l}_{h}$ is attained.
\end{lemma}

\begin{proof}[\bf{Proof}]
Firstly, we need to introduce some notations as in \cite{Yang 2018}. Set $|I_{h}|=M$. Define
$$
I^{\gamma}_{M}:=\inf_{J_{M}(\mathbf{u})=\gamma}I_{M}(\mathbf{u}), \quad \gamma_{G}:=\left(\inf_{J_{M}(\mathbf{u})=1}I_{M}(\mathbf{u})\right)^{2},
$$
where
$$
J_{M}(\mathbf{u}):=\sum_{(i,j)\in I_{h}^{2}}\int_{\mathbb{R}^{4}}\beta_{ij}u_i^2u_j^2, \quad
I_{M}(\mathbf{u}):=\sum_{i\in I_{h}}\int_{\mathbb{R}^{4}}|\nabla u_i|^2.
$$
It is clear that
\begin{equation}\label{2-2-1}
\gamma_{G}=(I^{1}_{M})^2=\gamma^{-1}(I^{\gamma}_{M})^2.
\end{equation}
We deduce from \eqref{Vector Sobolev Inequality-1} that $\gamma_{G}=4l_h$. It remains to prove that $I^{\gamma}_{M}$ is attained for some $\gamma>0$. Observe that the assumptions in \cite[Lemma 2.4]{Yang 2018} are true due to the fact that $\beta_{ij}\geq 0$  $\forall (i,j)\in I_{h}^{2}$ and $\beta_{ii}>0$; thus $I^{\gamma_{G}}_{M}$ is achieved. Therefore, ${l}_{h}$ is attained.
\end{proof}

\begin{lemma}\label{lemma:equality_l}
We have $l_h=\widetilde l_h$, and every minimizer for $\widetilde l_h$ is a minimizer for $l_h$.
\end{lemma}

\begin{proof}[\bf{Proof}]
This was proved in \cite[Lemma 5.2]{Tavares 2016-1}, but we sketch it here for completeness. We already know that $\tilde l_h\leq l_h$. If $\mathbf{v}\in \widetilde M_h$ then take $t^2:=(\int_\Omega  |\nabla \mathbf{v}|^2)/(\int_\Omega \sum_{(i,j)\in I_h^2} \beta_{ij} v_i^2 v_j^2)\leq 1 $, so that $t\mathbf{v}\in \mathcal{M}_h$. Then
\[
l_h \leq E_h(tv) = \frac{t^2}{4}\int_\Omega |\nabla \mathbf{v}|^2\leq \widetilde E_h(\mathbf{v}),
\]
and by taking the infimum for $\mathbf{v}\in \mathcal{M}_h$ we get $l_h\leq \widetilde l_h$ and so $l_h=\widetilde l_h$. If $\mathbf{v}$ is a minimizer for $\widetilde l_h$ (there exists at least one by what we have just seen combined with Lemma \ref{2.2-2}), then necessarily $t=1$ and $\mathbf{v}\in \mathcal{M}_h$.
\end{proof}

\begin{proof}[\bf Proof of Theorem \ref{limit system-6-1}.] The first part of the theorem follows from Lemmas \ref{2.2-2} and \ref{lemma:equality_l}. Now let $\mathbf{V}_h$ be a minimizer. By taking the absolute value, we may assume that it is nonnegative. By using the Lagrange multiplier rule and observing that $\mathcal{M}_h$ is a natural constraint, we see that $\mathbf{V}_h$ solves \eqref{sub-system}. Finally, once we know that the ground state level is achieved, we can follow the proof of \cite[Theorem 2.1]{Tavares 2016} (which is stated for subcritical problems, for $\lambda>0$ and $\beta_{ij}=b$ but works exactly in the same way in our framework) and deduce that $\mathbf{V}_h=X_0W$, where $W$ is a ground state solution of $-\Delta W=f_{max}W^3$. We can now conclude by taking the scaling $W=(f_{max}^h)^{-1/2}U_{\epsilon,y}$. Hence, $\mathbf{V}_{h}=X_0(f_{max}^h)^{-1/2}U_{\epsilon,y}$.
\end{proof}

The rest of the section is dedicated to the proof of Theorem \ref{limt system-6}. Having Lemmas \ref{2.2-2} and \ref{lemma:equality_l} at hand also in the critical case, we can now follow the strategy of \cite[Theorem 1.6]{Tavares 2016-1}. We will therefore sketch the proof, highlighting both the similarities as well as the differences. Take a nonnegative minimizer $\mathbf{V}_h=(V^h_i)_{i\in I_h}$ of $h=1,\ldots, m$. We need the following lemma.

\begin{lemma}\label{Decay 2.5}
Let $e_{1}\neq e_{2}\in \mathbb{S}^{N-1}$. Then, whenever $h_{1}\neq h_{2}$,
\begin{equation*}
\lim_{R\rightarrow +\infty}\int_{\mathbb{R}^{4}}\sum_{(i,j)\in I_{h_{1}}\times I_{h_{2}}}
\left(V_{i}^{h_{1}}(x-Re_{1})V_{j}^{h_{2}}(x-Re_{2})\right)^{2}=0.
\end{equation*}
\end{lemma}
\begin{proof}[\bf{Proof}]
For every $(i,j)\in I_{h_{1}}\times I_{h_{2}}$, we only need to prove that
\begin{equation*}
\lim_{R\rightarrow +\infty}\int_{\mathbb{R}^{4}}\left(V_{i}^{h_{1}}(x)V_{j}^{h_{2}}(x+Re_{1}-Re_{2})\right)^{2}=0.
\end{equation*}
Here we cannot argue as in \cite[Lemma 5.3]{Tavares 2016-1}  due to the decay of $V_h$ in the critical case. Set $V_{j,R}^{h_{2}}(x):=V_{j}^{h_{2}}(x+Re_{1}-Re_{2})$.
Since $e_{1}\neq e_{2}$, $R|e_{1}- e_{2}|\rightarrow +\infty$ as $R\rightarrow +\infty$, and each component of $V_h$ is a multiple of a bubble $U_{\epsilon, y}$ (by Theorem \ref{limit system-6-1}). Then $V_{j}^{h_{2}}(x+Re_{1}-Re_{2})$ converges almost everywhere to 0 as $R\to \infty$, and it is uniformly bounded in $L^4(\R^4)$. Therefore, by \cite[Proposition 5.4.7]{Willem 2013},   $V_{j}^{h_{2}}(x+Re_{1}-Re_{2})\rightharpoonup 0$ weakly in $L^{4}(\mathbb{R}^{4})$ as $R\rightarrow\infty$.
Hence,
\begin{equation*}
\lim_{R\rightarrow +\infty}\int_{\mathbb{R}^{4}}\left(V_{i}^{h_{1}}(x)V_{j,R}^{h_{2}}(x)\right)^{2}\leq
\lim_{R\rightarrow +\infty}\left(\int_{\mathbb{R}^{4}}(V_{i}^{h_{1}})^{3}V_{j,R}^{h_{2}}\right)^{\frac{2}{3}}
\left(\int_{\mathbb{R}^{4}}(V_{j,R}^{h_{2}})^{4}\right)^{\frac{1}{3}}=0. \qedhere
\end{equation*}
\end{proof}
\begin{proof}[\bf Proof of Theorem \ref{limt system-6}.]
 Firstly, we claim that $l=\sum_{h=1}^{m}l_{h}$. Based on Lemma \ref{Decay 2.5}, the proof is similar to those of Lemma 5.4 and Lemma 5.5 in \cite{Tavares 2016-1}, so we omit it. Now assume by contradiction that there exists $\mathbf{u}\in \mathcal{M}$ such that $E(\mathbf{u})=l$. By \cite[Lemma 2.3]{Tavares 2016-1} (which also holds in $\mathcal{D}^{1,2}(\mathbb{R}^4)$) we know that $\mathcal{M}$ is a natural constraint, and so $\mathbf{u}$ is a solution of \eqref{limit-system}. Moreover, we can suppose that $\mathbf{u}$ is nonnegative. Note that for every $h$ there exists $i_{h}$ such that $u_{i_{h}}\neq 0$; by the strong maximum principle, we see that $u_{i_{h}}>0$ in $\mathbb{R}^{4}$.

Observe that $\mathbf{u}\in \mathcal{M}$ and $\beta_{ij}\leq 0$ for every $(i,j)\in \mathcal{K}_{2}$, then we get that
\begin{equation*}
0<\|\mathbf{u}_{h}\|_{h}^{2}=\sum_{k=1}^{m}M_B(\mathbf{u})_{hk}\leq {M}_{B}(\mathbf{u})_{hh},
\end{equation*}
where $M_B(\mathbf{u})$ is defined at the end of Subsection \ref{subsec1.4}.

Then we see that $\mathbf{u}_{h}\in \widetilde{\mathcal{M}}_{h}$ for every $h=1,...,m$. Therefore,
\begin{equation}\label{2.6-1}
\frac{1}{4}\|\mathbf{u}_{h}\|_{h}^{2}\geq \inf_{\mathbf{v}\in\widetilde{\mathcal{M}}_{h}}\frac{1}{4}\|\mathbf{v}\|_{h}^{2}=\widetilde{l}_{h}=l_{h}, \text{ and }
E(\mathbf{u})=\sum_{h=1}^{m}\frac{1}{4}\|\mathbf{u}\|_{h}^{2}\geq \sum_{h=1}^{m}l_{h}.
\end{equation}
Combining these with $E(\mathbf{u})=l=\sum_{h=1}^{m}l_{h}$ we have
\begin{equation*}
\frac{1}{4}\|\mathbf{u}_{h}\|_{h}^{2}=l_{h}=\widetilde{l}_{h}.
\end{equation*}
That is, $\mathbf{u}_{h}$  minimizes for $\widetilde{l}_{h}$. By Lemma \ref{lemma:equality_l}, $\mathbf{u}_{h}$ is a minimizer for $l_{h}$ and $\mathbf{u}_{h}\in \mathcal{M}_{h}$. Hence,
\begin{equation*}
\|\mathbf{u}_{h}\|_{h}^{2}=M_{B}(\mathbf{u})_{hh}, ~\forall h=1,\ldots,m,
\end{equation*}
On the other hand, since $\mathbf{u}\in\mathcal{M}$, then we have
\begin{equation*}
\|\mathbf{u}_{h_{1}}\|_{h_{1}}^{2}=\sum_{k=1}^{m}\mathbf{M}_{B}(\mathbf{u})_{h_{1}k}\leq M_{B}(\mathbf{u})_{h_{1}h_{1}}+
\int_{\mathbb{R}^{4}}\beta_{i_{h_{1}}i_{h_{2}}}u^{2}_{i_{h_{1}}}u^{2}_{i_{h_{2}}}<M_{B}(\mathbf{u})_{h_{1}h_{1}},
\end{equation*}
which is a contradiction.
\end{proof}

%%%%%%
%%SECTION 3
%%%%%%%

\section{\bf Energy estimates and existence of Palais-Smale sequences}\label{sec3}
In this section we show crucial energy estimates, as well as other preliminary lemmas. All these results will be used to prove Theorem \ref{Theorem-1} in the next section.

Fix $\mathbf{a}=(a_1,\ldots, a_m)$, an $m$-decomposition of $d$, for some integer $m\in [1,d]$. Observe that this fixes the sets $\mathcal{K}_1$ and $\mathcal{K}_2$.
Throughout this section we always  assume \eqref{eq:coefficients}, \eqref{eq:beta_ij}, and so we omit these conditions in the statements.

To start with, we consider an analogue of the level $c$ where we just consider some components of the $m$ groups in which we divided $\{1,\ldots,d\}$. Given $\Gamma\subseteq \{1,\ldots,m\}$, we define
\begin{equation*}
J_\Gamma(\mathbf{u}):=\frac{1}{2}\sum_{k\in \Gamma}\|\mathbf{u}_{k}\|_{k}^{2}-\frac{1}{4}\sum_{k,l\in \Gamma}\sum_{(i,j)\in I_k\times I_l}\int_{\Omega}\beta_{ij}u_{i}^{2}u_{j}^{2},
\end{equation*}
\begin{equation}\label{eq:def_of_E}
E_\Gamma\left(\mathbf{u}\right):=\frac{1}{2}\sum_{k\in \Gamma}\sum_{i\in I_{k}}\int_{\mathbb{R}^{4}}|\nabla u_{i}|^{2}-\frac{1}{4}\sum_{k,l\in \Gamma}\sum_{(i,j)\in I_{l}\times I_{k}}\int_{\mathbb{R}^{4}}\beta_{ij}u_{j}^{2}u_{i}^{2},
\end{equation}
\begin{align}\label{Manifold-4.1}
\mathcal{N}_{\Gamma}:=
\Bigg\{(\mathbf{u}_h)_{h\in \Gamma}:    \mathbf{u}_{h} \neq \mathbf{0},\ \sum_{i\in I_{h}}
\partial_{i}J_\Gamma(\mathbf{u})u_{i}=0, \text{ for every } h\in \Gamma
\Bigg\},
\end{align}
\begin{equation}\label{Minimizer-4}
c_{\Gamma}=\inf_{\mathbf{u}\in\mathcal{N}_{\Gamma}}J_{\Gamma}(\mathbf{u}).
\end{equation}
Observe that $c=c_{\{1,\ldots,m\}}$.

\subsection{A uniform energy estimate and preliminary results}\label{subsec3.1}
\begin{lemma}\label{Pre-1}
Take
\begin{equation}\label{Pre-2}
\overline{C}=\frac{1}{4}\max_{ h\in \Gamma}\min_{i\in I_{h}}\left\{\frac{1}{\beta_{ii}}\right\}
\inf_{\Omega \supset\Omega_1,\ldots,\Omega_{m}\neq \emptyset \atop \Omega_{i}\cap\Omega_{j}=\emptyset, i\neq j}\sum_{h=1}^{m}\widetilde{S}^{2}(\Omega_{h}),
\end{equation}
where $\widetilde{S}(\Omega)$ is the best Sobolev constant for the embedding $H_{0}^{1}(\Omega)\hookrightarrow L^{4}(\Omega)$, defined by
\begin{equation*}
\widetilde{S}(\Omega):=\inf_{u\in H_{0}^{1}(\Omega)\setminus \{0\}}\frac{\int_{\Omega}|\nabla u|^{2}}{\left(\int_{\Omega}|u|^{4}\right)^{\frac{1}{2}}}.
\end{equation*}
Then
\begin{equation*}
c_\Gamma\leq \overline{C} \qquad \text{ for every } \Gamma\subseteq \{1,\ldots, m\}.
\end{equation*}
\end{lemma}
\begin{proof}[\bf{Proof}]
The following proof is inspired by that of Lemma 2.1 in \cite{Tavares 2016-1}.
For each $h\in \Gamma$, set $i_{h}$ the index attaining $\min_{i\in I_{h}}\{\frac{1}{\beta_{ii}}\}$. Take $\widehat{u}_{i_{1}},\ldots, \widehat{u}_{i_{|\Gamma|}}\not\equiv 0$ such that $\widehat{u}_{i_{h}}\cdot\widehat{u}_{i_{k}}\equiv 0$ whenever $h\neq k$. Denote $\widetilde{\mathbf{u}}$ by $\widetilde{u}_{i_{h}}=t_{h}\widehat{u}_{i_{h}}$, where $t_{h}=\|\widehat{u}_{i_{h}}\|_{i_{h}}/(\sqrt{\beta_{i_{h}i_{h}}}|\widehat{u}_{i_{h}}|^{2}_{L^{4}})$ for $h\in \Gamma$, and $\widetilde{u}_{i}=0$ for $i\neq i_{1},\ldots, i_{|\Gamma|}$. It is easy to see that $\widetilde{\mathbf{u}}_{h}\not \equiv 0$ for $h\in \Gamma$, and $\widetilde{\mathbf{u}}\in \mathcal{N}_\Gamma$. Thus, since $\lambda_i<0$, we infer that
\begin{align*}
 c_\Gamma & \leq J_\Gamma(\widetilde{\mathbf{u}})=\frac{1}{4}\sum_{h\in \Gamma}\|\widetilde{\mathbf{u}}_{h}\|^{2}_{h}=\frac{1}{4} \sum_{h\in \Gamma}t_{h}^{2}\|\widehat{u}_{i_{h}}\|_{i_{h}}^{2}\\
 & \leq\frac{1}{4}\sum_{h\in \Gamma}\frac{1}{\beta_{i_{h}i_{h}}}\frac{\|\widehat{u}_{i_{h}}\|_{H^{1}_0}^{4}}
{|\widehat{u}_{i_{h}}|^{4}_{L^{4}}}
\leq\frac{1}{4}\max_{ h\in \Gamma}\left\{\frac{1}{\beta_{i_{h}i_{h}}}\right\}\sum_{h\in \Gamma}\frac{\|\widehat{u}_{i_{h}}\|_{H^{1}_0}^{4}}
{|\widehat{u}_{i_{h}}|^{4}_{L^{4}}},
\end{align*}
where
\begin{equation*}
\|u\|_{H^{1}_0}^{2}:=\int_{\Omega}|\nabla u|^{2}.
\end{equation*}
Therefore
\begin{equation}\label{Constant-6-2}
c_\Gamma\leq \frac{1}{4}\max_{ h\in \Gamma}\min_{i\in I_{h}}\left\{\frac{1}{\beta_{ii}}\right\}
\inf_{\Omega \supset\Omega_1,\ldots,\Omega_{|\Gamma|}\neq \emptyset \atop \Omega_{i}\cap\Omega_{j}=\emptyset, i\neq j}\sum_{h=1}^{|\Gamma|}\widetilde{S}^{2}(\Omega_{h})\leq \frac{1}{4}\max_{ h\in \Gamma}\min_{i\in I_{h}}\left\{\frac{1}{\beta_{ii}}\right\}
\inf_{\Omega \supset\Omega_1,\ldots,\Omega_{m}\neq \emptyset \atop \Omega_{i}\cap\Omega_{j}=\emptyset, i\neq j}\sum_{h=1}^{m}\widetilde{S}^{2}(\Omega_{h}),
\end{equation}
which yields that $c_\Gamma\leq \overline{C}$, where $\overline{C}$ is defined in \eqref{Pre-2}.
\end{proof}
\begin{remark}\label{3-1}
In \cite[Lemma 2.1]{Tavares 2016-1} a related estimate was proved for the level $c$. Here, however, we will need estimates between the different levels $c_\Gamma$ (see Theorem \eqref{Energy Estimation-1,2,m-1}), and for that the uniform estimate independent of $\Gamma$ we have just obtained is crucial.
\end{remark}

Denote
\begin{equation}\label{Constant-6}
\Lambda_{1}:=S^{2}/(32\overline{C}),
\end{equation}
where $S$ is defined in  \eqref{eq:def_of_S}. From Lemma \ref{Pre-1} and  \cite[Lemma 2.2]{Tavares 2016-1} we have the following.
\begin{lemma}\label{Positive Definite}
Let $\Gamma\subseteq \{1,\ldots,m\}$. If
\begin{equation*}
\beta_{ij}\geq 0  \quad\forall (i,j)\in \mathcal{K}_{1} \quad and \quad
-\infty<\beta_{ij}< \Lambda_{1}  \quad\forall (i,j)\in \mathcal{K}_{2},
\end{equation*}
then there holds
\begin{equation*}
\mathcal{N}_\Gamma\cap \left\{(\mathbf{u}_h)_{h\in \Gamma}: \sum_{h\in \Gamma}\|\mathbf{u}_h\|_{h}^{2}\leq 8\overline{C}\right\}\subset \mathcal{E}_\Gamma,
\end{equation*}
where $\mathcal{E}_\Gamma$ is defined in \eqref{Diagonally Dominant}.
\end{lemma}

Repeating the proof of Proposition 1.2 in \cite{Soave 2015}, we have the following.
\begin{lemma}\label{Lagrange}
Let $\Gamma\subseteq \{1,\ldots, m\}$ and assume that $c_\Gamma$ is achieved by $\mathbf{u}\in \mathcal{N}_\Gamma\cap \mathcal{E}_\Gamma$. Then $\mathbf{u}$ is a critical point of $J_\Gamma$ with at least $|\Gamma|$ nontrivial components.
\end{lemma}

We finish this preliminary subsection with a lower and upper uniform estimate on the $L^4$-norms of elements in the Nehari set which are below a certain energy level.
\begin{lemma}\label{Bounded}
Let $\Gamma\subseteq \{1,\ldots, m\}$. If
\begin{equation*}
\beta_{ij}\geq 0  \quad\forall (i,j)\in \mathcal{K}_{1} \quad and \quad
-\infty<\beta_{ij}< \Lambda_{1}  \quad\forall (i,j)\in \mathcal{K}_{2},
\end{equation*}
then there exists $C_{2}>C_{1}>0$, such that for any $\mathbf{u}\in \mathcal{N}_\Gamma$ with $J_\Gamma(\mathbf{u})\leq 2\overline{C}$,
there holds
\begin{equation*}
C_{1}<\sum_{i\in I_{h}}|u_{i}|_{4}^{2}<C_{2}, ~h\in \Gamma,
\end{equation*}
where $C_{1},C_2$ are dependent only on $\beta_{ij}\geq 0$ for $(i,j)\in \mathcal{K}_{1}$, $\beta_{ii}, \lambda_{i}, i=1,\ldots, d$.
\end{lemma}
\begin{proof}[\bf{Proof}]
For any $\mathbf{u}\in \mathcal{N}_\Gamma$ with $J_\Gamma(\mathbf{u})\leq 2\overline{C}$ and for any $h\in \Gamma$, we have
\begin{equation*}
S\sum_{i\in I_{h}}|u_{i}|_{4}^{2}\leq\sum_{i\in I_{h}}\|u_{i}\|_{i}^{2} \leq\sum_{h\in \Gamma}\|\mathbf{u}_{h}\|_{h}^{2}
\leq 8\overline{C},
\end{equation*}
which yields that $\sum_{i\in I_{h}}|u_{i}|_{4}^{2}<C_{2}$. Similarly to Lemma 2.4 in \cite{Tavares 2016-1} we get that $\sum_{i\in I_{h}}|u_{i}|_{4}^{2}> C_{1}$.
\end{proof}

\subsection{Comparing least energy levels}\label{subsec3.2}
In this subsection we present crucial estimates both for $c$ and for all levels $c_\Gamma$, so that we can obtain in the next section that these levels are achieved. The following theorem plays a critical role in presenting that limits of minimizing sequences are nontrivial. The presence of a $\delta$ allows to consider also positive $\beta_{ij}\in \mathcal{K}_2$ in Theorem \ref{Theorem-1}.

\begin{thm}\label{Energy Estimates}
Assume that $\beta_{ij}\geq 0$ for $(i,j)\in \mathcal{K}_{1}$. There exists $\delta>0$, depending only on $\beta_{ij}\geq 0$ for $(i,j)\in \mathcal{K}_{1}$ and $\beta_{ii}, \lambda_{i} ~i=1,\ldots,d$, such that
\begin{equation*}
c_\Gamma<\sum_{h\in \Gamma}l_{h}-\delta \qquad \forall \Gamma\subseteq \{1,\ldots, m\},
\end{equation*}
where $c_\Gamma$ is defined in \eqref{Minimizer-4} and $l_{h}$ is defined in \eqref{eq:Ground State-2}. In particular,
\[
c<\sum_{h=1}^m l_h -\delta.
\]
\end{thm}
\begin{proof}[\bf{Proof}]
In order to reduce technicalities with indices, we present the proof for $\Gamma=\{1,\ldots, m\}$; the general case follows exactly in the same way. Take $y_{1}, ..., y_{m}\in \Omega$ such that
\begin{equation*}
0<4\rho:= \min_{1\leq h\neq k\leq m}|y_{h}-y_{k}|<\min_{1\leq h\leq m}dist(y_{h}, \partial\Omega).
\end{equation*}
Then $B_{2\rho}(y_{h})\subseteq \Omega$ and $B_{2\rho}(y_{h})\cap B_{2\rho}(y_{k})=\emptyset$ for $h\neq k$. Let $\xi_{h}\in C^{1}_{0}(B_{2\rho}(y_{h}))$ be such that $0\leq\xi_{h}\leq 1$ and $\xi_{h}\equiv 1$ for $|x-y_{h}|\leq\rho$, $h=1,...,m$.

Define
\begin{equation}\label{D-5}
\mathbf{V}_{h}^{\varepsilon}:=\varepsilon^{-1}\mathbf{V}_{h}\left(\frac{x-y_{h}}{\varepsilon}\right),
\end{equation}
where
\begin{equation*}
\mathbf{V}_{h}=X_0(f_{max}^h)^{-1/2}U_{1,0} ~\text{is a minimizer for $l_h$}
\end{equation*}
(recall Theorem \ref{limit system-6-1} and observe that $X_{0}, f_{max}^h$ are defined in \eqref{Class-1}) and so
\[
\mathbf{V}_{h}^{\varepsilon}=\varepsilon^{-1}X_0(f_{max}^h)^{-1/2}U_{1,0}\left(\frac{x-y_{h}}{\varepsilon}\right).
\]
Then we see that
\begin{equation*}
\int_{\mathbb{R}^{4}}|\nabla V_{i}^{\varepsilon}|^{2}=\int_{\mathbb{R}^{4}}|\nabla V_{i}|^{2}, \quad
\int_{\mathbb{R}^{4}}|V_{i}^{\varepsilon}|^{4}=\int_{\mathbb{R}^{4}}|V_{i}|^{4}, ~i\in I_{h}.
\end{equation*}
Define
\begin{equation}\label{Definition-2}
\widehat{\mathbf{V}}_{h}^{\varepsilon}:=\xi_{h}\mathbf{V}_{h}^{\varepsilon}.
\end{equation}
Therefore, by Lemma 1.46 in \cite{Willem 1996} we get the following inequalities

\begin{equation}\label{2.14}
\int_{\Omega}|\widehat{V}_{i}^{\varepsilon}|^{2}\geq C_{i}\varepsilon^{2}|ln\varepsilon|+O(\varepsilon^{2}), \text{ for every } i\in I_{h},
\end{equation}

\begin{equation}\label{2.12}
\int_{\Omega}|\nabla \widehat{V}_{i}^{\varepsilon}|^{2}= \int_{\mathbb{R}^{4}}|\nabla V_{i}|^{2}+O(\varepsilon^{2}), \text{ for every } i\in I_{h},
\end{equation}
where
\begin{equation}\label{Constant-2-9}
C_{i}=8(X_{0})_{i}^{2}(f_{max}^h)^{-1}.
\end{equation}
 Observe that all these quantities depend on $\beta_{ij}$ for $(i,j)\in I_h^2$. For every $(i,j)\in I_{h}^{2}$, we claim the following inequality
\begin{equation}\label{2.13}
\int_{\Omega}|\widehat{V}_{i}^{\varepsilon}|^{2}|\widehat{V}_{j}^{\varepsilon}|^{2}\geq \int_{\mathbb{R}^{4}}| V_{i}|^{2}| V_{j}|^{2}+O(\varepsilon^{4}).
\end{equation}
Observe that
\begin{align*}
\int_{\mathbb{R}^{4}}(1-\xi_{h}^{4})|\widehat{V}_{i}^{\varepsilon}|^{2}|\widehat{V}_{j}^{\varepsilon}|^{2}&\leq
\int_{|x|\geq \rho}\varepsilon^{-4}|V_{i}(x/\varepsilon)|^{2}|V_{j}(x/\varepsilon)|^{2}=\int_{|x|\geq \rho/\varepsilon}|V_{i}|^{2}|V_{j}|^{2} \\
& =\int_{|x|\geq \rho/\varepsilon}\frac{C}{(1+|x|^2)^4}\leq\int_{|x|\geq \rho/\varepsilon} \frac{C}{|x|^8}
=O(\varepsilon^{4}),
\end{align*}
where $C$ is only dependent on $\beta_{ij}$ for $(i,j)\in I_h^2$.
Hence
\begin{align*}
\int_{\Omega}|\widehat{V}_{i}^{\varepsilon}|^{2}|\widehat{V}_{j}^{\varepsilon}|^{2}&=\int_{\mathbb{R}^{4}}|{V}_{i}^{\varepsilon}|^{2}|{V}_{j}^{\varepsilon}|^{2} -\int_{\mathbb{R}^{4}}(1-\xi_{h}^{4})|\widehat{V}_{i}^{\varepsilon}|^{2}|\widehat{V}_{j}^{\varepsilon}|^{2}
\geq \int_{\mathbb{R}^{4}}| V_{i}|^{2}| V_{j}|^{2}+O(\varepsilon^{4}),
\end{align*}
which implies that \eqref{2.13} holds.

Recall that $\lambda_{1},\cdots,\lambda_{d}<0$, $\beta_{ij}\geq 0$ for every $(i,j)\in \mathcal{K}_{1}$ and $\beta_{ii}>0$. Note that $\int_{\Omega}|\widehat{V}_{i}^{\varepsilon}|^{2}|\widehat{V}_{j}^{\varepsilon}|^{2}dx=0$ for every $(i,j)\in \mathcal{K}_{2}$. Then we deduce from \eqref{2.14}-\eqref{2.12} that, given $t_1,\ldots, t_m>0$,
\begin{align}\label{2.15}
J \left(\sqrt{t_{1}}\widehat{\mathbf{V}}_{1}^{\varepsilon},..., \sqrt{t_{m}}\widehat{\mathbf{V}}_{m}^{\varepsilon}\right)&=
\frac{1}{2}\sum_{h=1}^{m}t_{h}\sum_{i\in I_{h}}\int_{\Omega}|\nabla \widehat{V}_{i}^{\varepsilon}|^{2}+\lambda_{i}| \widehat{V}_{i}^{\varepsilon}|^{2}
 -\frac{1}{4}\sum_{h=1}^{m}t_{h}^{2}\sum_{(i,j)\in I_{h}^{2}}\int_{\Omega}\beta_{ij}|\widehat{V}_{i}^{\varepsilon}|^{2}
|\widehat{V}_{j}^{\varepsilon}|^{2}\nonumber\\
&\leq
\frac{1}{2}\sum_{h=1}^{m}t_{h}\left(\int_{\mathbb{R}^{4}}|\nabla \mathbf{V}_{h}|^{2}-C^h\varepsilon^{2}|ln\varepsilon|+ O(\varepsilon^{2})\right) \nonumber\\
& \quad -\frac{1}{4}\sum_{h=1}^{m}t_{h}^{2}\left(\sum_{(i,j)\in I_{h}^{2}}\int_{\mathbb{R}^{4}}\beta_{ij}|V_{i}|^{2}|V_{j}|^{2}
+O(\varepsilon^{4})\right) \nonumber\\
&= \frac{1}{2}\sum_{h=1}^{m}t_{h}\left(4l_{h}-C^h\varepsilon^{2}|ln\varepsilon|+ O(\varepsilon^{2})\right)-\frac{1}{4}\sum_{h=1}^{m}t_{h}^{2}\left(4l_{h}+O(\varepsilon^{4})\right),
\end{align}
where $C^h:= \sum_{i\in I_h}C_i|\lambda_{i}|=\sum_{i\in I_h}8(X_{0})_{i}^{2}(f_{max}^h)^{-1}|\lambda_{i}|>0$ and $C_i$ is defined in \eqref{Constant-2-9}.
Denote
\begin{equation*}
A_{h}^{\varepsilon}:=4l_{h}-C^h\varepsilon^{2}|ln\varepsilon|+ O(\varepsilon^{2}),\quad
B_{h}^{\varepsilon}:=4l_{h}+O(\varepsilon^{4}).
\end{equation*}
It is easy to see that
\begin{equation*}
0<A_{h}^{\varepsilon}<B_{h}^{\varepsilon}, \text{ for } \varepsilon \text{ small enough. }
\end{equation*}
Combining these we have
\begin{align}\label{2.15-3}
\max_{t_{1},\ldots,t_{m}>0}J \left(\sqrt{t_{1}}\widehat{\mathbf{V}}_{1}^{\varepsilon},..., \sqrt{t_{m}}\widehat{\mathbf{V}}_{m}^{\varepsilon}\right)
&\leq\max_{t_{1},\ldots,t_{m}>0}\sum_{h=1}^{m}\left( \frac{1}{2}t_{h}A_{h}^{\varepsilon}- \frac{1}{4}t_{h}^{2}B_{h}^{\varepsilon}\right)= \frac{1}{4}\sum_{h=1}^{m}\frac{\left(A_{h}^{\varepsilon}\right)^{2}}{B_{h}^{\varepsilon}} \nonumber\\
&< \frac{1}{4}\sum_{h=1}^{m}A_{h}^{\varepsilon}=\sum_{h=1}^{m}\left(l_{h}-\frac{1}{4}C^h\varepsilon^{2}|ln\varepsilon|+ O(\varepsilon^{2})\right)\nonumber\\
&< \sum_{h=1}^{m}l_{h}-\delta ~\text{ for } \varepsilon \text{ small enough, }
\end{align}
where
\begin{equation}\label{4-p-9}
\delta:=\frac{1}{16}\min_{1\leq h\leq m}\{C^h\}\varepsilon^{2}|ln\varepsilon|,
\end{equation}
$\delta$ being a positive constant, only dependent on $\beta_{ij}\geq 0$ for $(i,j)\in \mathcal{K}_{1}$ and $\beta_{ii}, \lambda_{i} ~i=1,\ldots, d$.
Note that the matrix $M_{B}(\widehat{\mathbf{V}}^{\varepsilon})$ is diagonal, so it is easy to see that there exists $t_{1}^{\varepsilon},...,t_{m}^{\varepsilon}>0$ such that
\begin{equation*}
\left(\sqrt{t_{1}^{\varepsilon}}\widehat{\mathbf{V}}_{1}^{\varepsilon},..., \sqrt{t_{m}^{\varepsilon}}\widehat{\mathbf{V}}_{m}^{\varepsilon}\right)\in \mathcal{N}.
\end{equation*}
Thus,
\begin{equation*}
c\leq J\left(\sqrt{t_{1}^{\varepsilon}}\widehat{\mathbf{V}}_{1}^{\varepsilon},..., \sqrt{t_{m}^{\varepsilon}}\widehat{\mathbf{V}}_{m}^{\varepsilon}\right)\leq \max_{t_{1},\ldots,t_{m}>0}J \left(\sqrt{t_{1}}\widehat{\mathbf{V}}_{1}^{\varepsilon},..., \sqrt{t_{m}}\widehat{\mathbf{V}}_{m}^{\varepsilon}\right)
<\sum_{h=1}^{m}l_{h}-\delta,
\end{equation*}
which yields that $c<\sum_{h=1}^{m}l_{h}-\delta$.
\end{proof}

The following is an extension of Theorem \ref{Energy Estimates}. It compares the energy of a level $c_\Gamma$ with all the levels $c_G$ with $G\subsetneq \Gamma$, assuming the later are achieved. In the next section this will play a crucial role in proving (via mathematical induction in the number of sub-groups) that $c$ is achieved by a solution with $m$ nontrival components. Define
\begin{equation}\label{Constant-4-1}
\Lambda_{2}:= \frac{S^{2}}{16(\sum_{h=1}^{m}l_{h}-2\delta)} ~~\text{ and } \Lambda_{3}:=\min\{\Lambda_{1}, \Lambda_{2}\},
\end{equation}
where $l_{h}$ is defined in \eqref{eq:Ground State-2}, $h=1,\ldots,m$ and $\delta$  is defined in Theorem \ref{Energy Estimates}. Then we have the following result.

\begin{thm}\label{Energy Estimation-1,2,m-1}
Take
\begin{equation*}
\beta_{ij}\geq 0  \quad\forall (i,j)\in \mathcal{K}_{1} ~and~
-\infty<\beta_{ij}< \Lambda_{3} \quad\forall (i,j)\in \mathcal{K}_{2}.
\end{equation*}
Given $\Gamma\subseteq\{1,\ldots, m\}$, assume that
\[
c_G \text{ is achieved by a nonnegative } \mathbf{u}_G \text{ for every } G\subsetneq \Gamma.
\]
Then
\begin{equation*}
c_\Gamma<\min\left\{  c_G + \sum_{h\in \Gamma\setminus G } l_h-\delta:\ G\subsetneq \Gamma \right\},
\end{equation*}
where $\delta$ is defined in Theorem \ref{Energy Estimates}, depending only on $\beta_{ij}\geq 0$ for $(i,j)\in \mathcal{K}_{1}$ and $\beta_{ii}, \lambda_{i} ~i=1,\ldots, d$. (and not depending on $\Gamma$).
\end{thm}

The rest of this section is dedicated to the proof of this theorem. Without loss of generality, we fix $1\leq p<m$ and prove that
\begin{equation}\label{eq:c<c_p}
c<c_{1,\ldots,p}+\sum_{h=p+1}^{m}l_{h}-\delta
\end{equation}
(where we use the notation $c_{1,\ldots, p}$ instead of $c_{\{1,\ldots,p\}}$ for simplicity); the other inequalities follow in the same way. The important fact that we like to stress is that  $\delta$ does not depend on $\Gamma$ nor on $\beta_{ij}$ for $(i,j)\in \mathcal{K}_2$; in order to highlight that, we always exhibit the explicit dependences of the constants.

Let $-\infty<\beta_{ij}< \Lambda_{1} ~\forall (i,j)\in \mathcal{K}_{2}$. Suppose that $c_{1,\ldots,p}$ is attained by a nonnegative $\mathbf{u}^{p}=(\mathbf{u}_{1},\cdots, \mathbf{u}_{p})$. Then by Lemma \ref{Pre-1}, Lemma \ref{Positive Definite} and Lemma \ref{Lagrange} we infer that $\mathbf{u}^{p}$ is a solution of the corresponding subsystem. By theory of elliptic regularity we get that $u_{i}\in C^{2}(\overline{\Omega})$.
Observe that $u_{i}\equiv 0$ on $\partial\Omega$, $i\in I^{p}:=I_1\cup\ldots
\cup I_p$. Taking $x_{1}\in \partial\Omega$, there exists $\rho_{1}>0$
such that
\begin{align}\label{3-p-23}
 \Pi^{2}&:= \max_{i\in I^{p}}\sup_{x\in B_{2\rho_{1}}(x_{1})\cap \Omega}u_{i}^{2}(x)\leq \min\Bigg\{ \frac{\min_{1\leq i\leq d}\{|\lambda_{i}|\}}{2d\max_{(i,j)\in \mathcal{K}_{2}}\{|\beta_{ij}|\}\widehat{t}}, \nonumber\\
 & \frac{SC_{1}}{4d^{3}\widetilde{C}\max_{(i,j)\in \mathcal{K}_{2}}\{|\beta_{ij}|\}}, ~~
 \frac{\min_{1\leq h\leq m}\{l_{h}\}}{d^{3}\widetilde{C}\max_{(i,j)\in \mathcal{K}_{2}}\{|\beta_{ij}|\}}\Bigg\},
\end{align}
where $C_{1}$ is defined in Lemma \ref{Bounded}, $\widetilde{C}:=\sqrt{\frac{4\max_h\{l_h\} |\Omega|}{\min_i\{\beta_{ii}\}}}$, and
\begin{equation}\label{eq:t_and_theta}
\widehat{t}= \frac{8\max\{\overline{C},l_1,\ldots, l_m \}}{\theta},\qquad \text{ where } \theta:=\min\left\{\frac{S}{4}C_{1}, l_{1},\ldots, l_{m}\right\}.
\end{equation}
Therefore, there exists $\rho>0$ such that $B_{2\rho}(x_{0})\subseteq B_{2\rho_{1}}(x_{1})\cap \Omega$, and so
\begin{equation}\label{3-p-16}
\sup_{x\in B_{2\rho}(x_{0})}u_{i}^{2}(x)\leq \Pi^{2} ~\forall i\in I^{p}.
\end{equation}

Take $x_{p+1},\cdots, x_{m}\in B_{2\rho}(x_{0})$ and $\rho_{p+1},\cdots,\rho_{m}>0$ such that $B_{2\rho_{p+1}}(x_{p+1}), \cdots, B_{2\rho_{m}}(x_{m})\subset B_{2\rho}(x_{0})$ and $B_{2\rho_i}(x_i)\cap B_{2\rho_j}(x_j)=\emptyset$ for every $i\neq j$, $i,j\in \{p+1,\ldots, m\}$. Let $\xi_{h}\in C^{1}_{0}(B_{2\rho_{h}}(x_{h}))$ be a nonnegative function with $\xi_{h}\equiv 1$ for $|x-x_{h}|\leq \rho_{h}, h=p+1,\ldots,m$.
Define
\begin{equation}\label{3-p-21}
\mathbf{v}_{h}^{\varepsilon}:=\xi_{h}\mathbf{V}^{\varepsilon}_{h},
\end{equation}
where
\begin{equation*}
\mathbf{V}_{h}^{\varepsilon}:=  \varepsilon^{-1}\mathbf{V}_{h}\left(\frac{x-y_{h}}{\varepsilon}\right)= \varepsilon^{-1}X_0(f_{max}^h)^{-1/2}U_{1,0}\left(\frac{x-y_{h}}{\varepsilon}\right), \quad \mathbf{V}_{h}=X_0(f_{max}^h)^{-1/2}U_{1,0} ~\text{ a minimizer for $l_h$}.
\end{equation*}
Similarly to \eqref{2.14}-\eqref{2.12} from the proof of the previous theorem, we have the following inequalities
\begin{equation}\label{2.12-1,2-m}
\int_{\Omega}|\nabla v_{i}^{\varepsilon}|^{2}= \int_{\mathbb{R}^{4}}|\nabla V_{i}|^{2}+O(\varepsilon^{2}),\text{ for every } i\in I_{h}, ~h=p+1,\ldots,m,
\end{equation}
\begin{equation}\label{2.13-1,2-m}
\int_{\Omega}|v_{i}^{\varepsilon}|^{2}|v_{j}^{\varepsilon}|^{2}\geq \int_{\mathbb{R}^{4}}| V_{i}|^{2}| V_{j}|^{2}+O(\varepsilon^{4}),\text{ for every } (i,j)\in I_{h}^{2}, ~h=p+1,\ldots,m.
\end{equation}
Moreover, we have
\begin{align}
\int_{\Omega}|v_{i}^{\varepsilon}|^{2} &\leq |\Omega|^{\frac{1}{2}} \left(\int_{\Omega}|v_{i}^{\varepsilon}|^{4}\right)^{\frac{1}{2}}
\leq  |\Omega|^{\frac{1}{2}} \left(\int_{\mathbb{R}^{4}}|V_{i}|^{4}\right)^{\frac{1}{2}}\nonumber\\
&=\frac{|\Omega|^{\frac{1}{2}}}{\beta_{ii}^\frac{1}{2}} \left(\int_{\mathbb{R}^{4}}\beta_{ii}|V_{i}|^{4}\right)^{\frac{1}{2}}
\leq \frac{|\Omega|^{\frac{1}{2}}}{\min\{\beta_{ii}^\frac{1}{2}\}} \left(\sum_{(i,j)\in I_{h}^{2}}\int_{\mathbb{R}^{4}}\beta_{ij}V_{i}^{2}V_{j}^{2}\right)^{\frac{1}{2}}\nonumber\\
&\leq \sqrt{\frac{4\max_h\{l_h\} |\Omega|}{\min_i\{\beta_{ii}\}}}=\widetilde{C}.
\end{align}
Combining this with \eqref{2.14}, we have
\begin{equation}\label{2.14-1,2-m}
C\varepsilon^{2}|ln\varepsilon|+O(\varepsilon^{2})\leq\int_{\Omega}|v_{i}^{\varepsilon}|^{2}\leq \widetilde{C}, \text{ for every } i\in I_{h},
\end{equation}
where we remark that $C, \widetilde{C}>0$ are dependent only on $\beta_{ii}>0$, $\beta_{ij}\geq 0$ for $(i,j)\in \mathcal{K}_{1}$.

Let us now explain the idea of the proof of \eqref{eq:c<c_p} (and so, of Theorem \ref{Energy Estimation-1,2,m-1}).
Consider
\begin{align}\label{3-p-4}
\Phi(t_{1},\cdots,t_{m}):&=J(\sqrt{t_{1}}\mathbf{u}_{1},\cdots, \sqrt{t_{p}}\mathbf{u}_{p}, \sqrt{t_{p+1}}\mathbf{v}_{p+1}^{\varepsilon},\cdots,\sqrt{t_{m}}\mathbf{v}_{m}^{\varepsilon})\nonumber\\
& =\frac{1}{2}\sum_{k=1}^{p}t_{k}\|\mathbf{u}_{k}\|_{k}^{2}
+\sum_{h=p+1}^{m}\frac{1}{2}t_{h}\sum_{i\in I_{h}}\int_{\Omega}|\nabla v^{\varepsilon}_{i}|^{2}+\lambda_{i}| v^{\varepsilon}_{i}|^{2}\, dx-\frac{1}{4}M_{B}(\mathbf{u}^{\varepsilon})\mathbf{t}\cdot\mathbf{t},
\end{align}
where
\begin{equation}\label{3-p-6-1}
\mathbf{u}^{\varepsilon}:= (\mathbf{u}_{1}^{\varepsilon},\cdots, \mathbf{u}_{m}^{\varepsilon})=(\mathbf{u}^{p}, \mathbf{v}_{p+1}^{\varepsilon},\cdots, \mathbf{v}_{m}^{\varepsilon}),
\end{equation}
$\mathbf{u}^{p}=(\mathbf{u}_{1},\cdots, \mathbf{u}_{p})$ is a vector with nonnegative components attaining $c_{1,\ldots,p}$, and $\mathbf{v}_{h}^{\varepsilon}$ is defined in \eqref{3-p-21}.
The strategy is to prove that
\begin{equation}\label{3-p-3-18}
\max_{t_{1},\cdots,t_{m}>0} J\left(\sqrt{t_{1}}\mathbf{u}_{1},\cdots, \sqrt{t_{p}}\mathbf{u}_{p},\sqrt{t_{p+1}}\mathbf{v}_{p+1}^{\varepsilon},\cdots,\sqrt{t_{m}}\mathbf{v}_{m}^{\varepsilon}\right)<
c_{1,\ldots,p}+\sum_{h=p+1}^{m}l_{h}-\delta,
\end{equation}
and that there exists $t_{1}^{\varepsilon},\ldots,t_{p}^{\varepsilon},\ldots,t_{m}^{\varepsilon}>0$ such that
\begin{equation}\label{3-p-3-4}
\left(\sqrt{t_{1}^{\varepsilon}}\mathbf{u}_{1},\cdots,\sqrt{t_{p}^{\varepsilon}}\mathbf{u}_{p}, \sqrt{t_{p+1}^{\varepsilon}}\mathbf{v}_{p+1}^{\varepsilon},\cdots,\sqrt{t_{m}^{\varepsilon}}\mathbf{v}_{m}^{\varepsilon}\right)\in \mathcal{N}.
\end{equation}
Obviously \eqref{eq:c<c_p} is a consequence of \eqref{3-p-3-18} and \eqref{3-p-3-4}.

Firstly we show that $M_{B}(\mathbf{u}^{\varepsilon})$ is positive definite.
\begin{lemma}\label{Con-4-1}
If
\begin{equation*}
\beta_{ij}\geq 0  \quad\forall (i,j)\in \mathcal{K}_{1} ~and~
-\infty<\beta_{ij}< \Lambda_{1} \quad\forall (i,j)\in \mathcal{K}_{2}
\end{equation*}
then $M_{B}(\mathbf{u}^{\varepsilon})$ is strictly diagonally dominant and $M_{B}(\mathbf{u}^{\varepsilon})$ is
positive definite for small $\varepsilon$. Moreover, we can get that
\begin{equation}\label{3-p-7-3}
\kappa_{min}\geq \theta,
\end{equation}
where $\kappa_{min}$ is the minimum eigenvalues of $M_{B}(\mathbf{u}^{\varepsilon})$ and $\theta$ is defined in \eqref{eq:t_and_theta}.
\end{lemma}
\begin{proof}[\bf{Proof}] We need to check that
\begin{equation}\label{eq:diagonallydominant}
\sum_{(i,j)\in I_{k}^{2}}\int_{\Omega}\beta_{ij}(u_{i}^\varepsilon)^{2}(u_{j}^\varepsilon)^{2} -\sum_{l=1,l\neq k}^{m}
\left|\sum_{(i,j)\in I_{k}\times I_{l}}\int_{\Omega}\beta_{ij} (u_{i}^\varepsilon)^{2}(u_{j}^\varepsilon)^{2}\right|\geq \theta \qquad k=1,\ldots,m
\end{equation}
(recall that $\beta_{ij}\geq 0$ for every $(i,j)\in \mathcal{K}_1$). We separate the proof in two cases: first for $k=1,\ldots, p$, and then for $k=p+1,\ldots, m$.

\noindent \textbf{Case 1.} Firstly, we claim that
\begin{equation}\label{3-p-7}
\sum_{(i,j)\in I_{k}^{2}}\int_{\Omega}\beta_{ij}u_{i}^{2}u_{j}^{2} -\sum_{l=1,l\neq k}^{p}
\left|\sum_{(i,j)\in I_{k}\times I_{l}}\int_{\Omega}\beta_{ij} u_{i}^{2}u_{j}^{2}\right|\geq \frac{S}{2}C_{1}, ~k=1,\ldots, p,
\end{equation}
where $C_{1}$ is defined in Lemma \ref{Bounded}.
In fact, without loss of generality, there exists $\overline{m}_{1}\in \{0,\ldots,p\}$ such that
\begin{align}\label{Matrix-6-3}
\sum_{l=1,l\neq k}^{p}
\left|\sum_{(i,j)\in I_{k}\times I_{l}}\int_{\Omega}\beta_{ij} u_{i}^{2}u_{j}^{2}\right| =
-\sum_{l=1,l\neq k}^{\overline{m}_{1}}\sum_{(i,j)\in I_{k}\times I_{l}}\int_{\Omega}\beta_{ij} u_{i}^{2}u_{j}^{2}
 +\sum_{l=\overline{m}_{1}+1,l\neq k}^{p}\sum_{(i,j)\in I_{k}\times I_{l}}\int_{\Omega}\beta_{ij}u_{i}^{2}u_{j}^{2}.
\end{align}
Note that $\mathbf{u}^{p}\in\mathcal{N}_{1,\ldots,p}$, then we have
\begin{equation}\label{Matrix-5-3}
\sum_{(i,j)\in I_{k}^{2}}\int_{\Omega}\beta_{ij}u_{i}^{2}u_{j}^{2}=\sum_{i\in I_{k}}\|u_{i}\|_{i}^{2}-
\sum_{l=1,l\neq k}^{p}\sum_{(i,j)\in I_{k}\times I_{l}}\int_{\Omega}\beta_{ij}u_{i}^{2}u_{j}^{2}.
\end{equation}
Combining this with \eqref{Matrix-6-3} we know that
\begin{align}\label{Matrix-7-3}
\sum_{(i,j)\in I_{k}^{2}}\int_{\Omega}\beta_{ij}u_{i}^{2}u_{j}^{2} & -\sum_{l=1,l\neq k}^{p}
\left|\sum_{(i,j)\in I_{k}\times I_{l}}\int_{\Omega}\beta_{ij} u_{i}^{2}u_{j}^{2}\right|=\nonumber\\
& \sum_{i\in I_{k}}\|u_{i}\|_{i}^{2}
 -2\sum_{l=\overline{m}_{1}+1,l\neq k}^{p}\sum_{(i,j)\in I_{k}\times I_{l}}\int_{\Omega}\beta_{ij}u_{i}^{2}u_{j}^{2}.
\end{align}
We deduce from Lemma \ref{Pre-1} that $\sum_{k=1}^{p}\|\mathbf{u}_{k}\|_{k}^{2}=4c_{1,\ldots,p}<4\overline{C}$, then we have
\begin{align}\label{Matrix-7-1-3}
2\sum_{l=\overline{m}_{1}+1,l\neq k}^{p}\sum_{(i,j)\in I_{k}\times I_{l}}\int_{\Omega}\beta_{ij}u_{i}^{2}u_{j}^{2}&<
\frac{2\Lambda_{2}}{S^{2}} \sum_{l=\overline{m}_{1}+1,l\neq k}^{p}\sum_{(i,j)\in I_{k}\times I_{l}}\|u_{i}\|_{i}^{2}\|u_{j}\|_{j}^{2}\nonumber\\
& \leq \frac{8\Lambda_{2}\overline{C}}{S^{2}}\sum_{i\in I_{k}}\|u_{i}\|_{i}^{2}<\frac{1}{2}\sum_{i\in I_{k}}\|u_{i}\|_{i}^{2}.
\end{align}
We see from \eqref{Matrix-7-3}-\eqref{Matrix-7-1-3} and Lemma \ref{Bounded} that
\begin{align}\label{Matrix-2}
\sum_{(i,j)\in I_{k}^{2}}\int_{\Omega}\beta_{ij}u_{i}^{2}u_{j}^{2} & -\sum_{l=1,l\neq k}^{p}
\left|\sum_{(i,j)\in I_{k}\times I_{l}}\int_{\Omega}\beta_{ij} u_{i}^{2}u_{j}^{2}\right|\nonumber\\
& \geq\frac{1}{2}\sum_{i\in I_{k}}\|u_{i}\|_{i}^{2}\geq \frac{S}{2}\sum_{i\in I_{k}}|u_{i}|_{4}^{2}\geq \frac{S}{2}C_{1}, ~k=1,\ldots, p.
\end{align}
For $k=1,\ldots, p$ and  $h=p+1,\ldots, m$, by \eqref{3-p-16} and \eqref{2.14-1,2-m} we see that
\begin{align}\label{3-p-8}
\left|\sum_{(i,j)\in I_{k}\times I_{h}}
\int_{\Omega}\beta_{ij}u_{i}^{2}(v^{\varepsilon}_{j})^{2}\right|& \leq \max_{(i,j)\in \mathcal{K}_{2}}\{|\beta_{ij}|\} \Pi^{2}\sum_{(i,j)\in I_{k}\times I_{h}}
\int_{B_{2\rho_{h}}(x_{h})}|v^{\varepsilon}_{j}|^{2}\nonumber\\
& \leq d^{2}\widetilde{C}\max_{(i,j)\in \mathcal{K}_{2}}\{|\beta_{ij}|\} \Pi^{2},
\end{align}
where $\widetilde{C}$ is defined in \eqref{2.14-1,2-m}. It follows from \eqref{3-p-23}, \eqref{Matrix-2}-\eqref{3-p-8} that
\begin{align}\label{3-p-10}
\sum_{(i,j)\in I_{k}^{2}}\int_{\Omega}\beta_{ij}u_{i}^{2}u_{j}^{2} &-\sum_{l=1,l\neq k}^{p}
\left|\sum_{(i,j)\in I_{k}\times I_{l}}\int_{\Omega}\beta_{ij} u_{i}^{2}u_{j}^{2}\right|-\sum_{l=p+1,l\neq k}^{m}\left|\sum_{(i,j)\in I_{k}\times I_{l}}
\int_{\Omega}\beta_{ij}|u_{i}|^{2}|v^{\varepsilon}_{j}|^{2}\right|\nonumber\\
& \geq\frac{S}{2}C_{1}-d^{3}\widetilde{C}\max_{(i,j)\in \mathcal{K}_{2}}\{|\beta_{ij}|\}\Pi^{2}\geq\frac{S}{4}C_{1}\geq \theta,  ~k=1,\ldots, p.
\end{align}

\noindent \textbf{Case 2.} For every $k=p+1,\ldots,m$, we see from \eqref{2.13-1,2-m} that
\begin{equation}\label{2.13-1,2-p}
\sum_{(i,j)\in I_{k}^{2}}\int_{\Omega}\beta_{ij}|v_{i}^{\varepsilon}|^{2}|v_{j}^{\varepsilon}|^{2}\geq \sum_{(i,j)\in I_{k}^{2}}\int_{\mathbb{R}^{4}}\beta_{ij}| V_{i}|^{2}| V_{j}|^{2}+O(\varepsilon^{4})\geq 2l_{k} \text{ for } \varepsilon \text{ small enough}.
\end{equation}
Combining this with \eqref{3-p-23} and \eqref{3-p-8}, since $u_i^\varepsilon\cdot u_j^\varepsilon=v_i^\varepsilon\cdot v_j^\varepsilon=0$ whenever $i\in I_k,j\in I_l$, $k\neq l$, $k,l\in \{p+1,\ldots, m\}$, we see that
\begin{align}\label{3-p-11}
\sum_{(i,j)\in I_{k}^{2}}\int_{\Omega}\beta_{ij}| u^{\varepsilon}_{i}|^{2}|u^{\varepsilon}_{j}|^{2} &-\sum_{l=1,l\neq k}^{m}\left|\sum_{(i,j)\in I_{k}\times I_{l}}
\int_{\Omega}\beta_{ij}|u^{\varepsilon}_{i}|^{2}|u^{\varepsilon}_{j}|^{2}\right|\nonumber\\
&=\sum_{(i,j)\in I_{k}^{2}}\int_{\Omega}\beta_{ij}| v^{\varepsilon}_{i}|^{2}|v^{\varepsilon}_{j}|^{2} -\sum_{l=1,l\neq k}^{p}\left|\sum_{(i,j)\in I_{k}\times I_{l}}
\int_{\Omega}\beta_{ij}|u_{i}|^{2}|v^{\varepsilon}_{j}|^{2}\right|\nonumber\\
& \geq 2l_{k}-d^{3}\widetilde{C}\max_{(i,j)\in \mathcal{K}_{2}}\{|\beta_{ij}|\} \Pi^{2}
\geq l_{k}\geq \theta,  ~k=p+1,\ldots, m.
\end{align}

We deduce from \eqref{3-p-10} and \eqref{3-p-11} that $M_{B}(\mathbf{u}^{\varepsilon})$ is strictly diagonally dominant, and so $M_{B}(\mathbf{u}^{\varepsilon})$ is positive definite.

For any eigenvalue $\kappa$ of $M_{B}(\mathbf{u}^{\varepsilon})$, from  \eqref{eq:diagonallydominant} and by the Gershgorin circle theorem we see that $\kappa\geq \theta$. Thus, $\kappa_{min}\geq \theta$.
\end{proof}

Based upon the former lemma, we can now prove that \eqref{3-p-3-18} holds.

\begin{lemma}\label{Con-4-2}
Assume that
\begin{equation*}
\beta_{ij}\geq 0  \quad\forall (i,j)\in \mathcal{K}_{1} ~and~
-\infty<\beta_{ij}< \Lambda_{1} \quad\forall (i,j)\in \mathcal{K}_{2}.
\end{equation*}
Then we have
\begin{equation}\label{eq:maximum_estimate}
\max_{t_{1},\cdots,t_{m}>0} J\left(\sqrt{t_{1}}\mathbf{u}_{1},\cdots, \sqrt{t_{p}}\mathbf{u}_{p},\sqrt{t_{p+1}}\mathbf{v}_{p+1}^{\varepsilon},\cdots,\sqrt{t_{m}}\mathbf{v}_{m}^{\varepsilon}\right)<
c_{1,\ldots,p}+\sum_{h=p+1}^{m}l_{h}-\delta
\end{equation}
for small $\varepsilon$, where $\delta$ is as in Theorem \ref{Energy Estimates} and $c_{1,\ldots,p}$ is defined in \eqref{Minimizer-4}.
\end{lemma}
\begin{proof}[\bf{Proof}]
Recalling the definition of $\Phi$ (see \eqref{3-p-4}), it follows from Lemma \ref{Con-4-1} that
\begin{align}\label{3-p-4.2}
\Phi(t_{1},\cdots,t_{m})
 \leq\frac{1}{2}\sum_{k=1}^{p}t_{k}\|\mathbf{u}_{k}\|_{k}^{2}
+\sum_{h=p+1}^{m}\frac{1}{2}t_{h}\sum_{i\in I_{h}}\int_{\Omega}|\nabla v^{\varepsilon}_{i}|^{2}+\lambda_{i}| v^{\varepsilon}_{i}|^{2}\, dx-\frac{1}{4}\theta\sum_{k=1}^{m}t^{2}_{k}.
\end{align}
Similarly to \eqref{2.15} we infer that
\begin{align}\label{3-p-4.3}
\sum_{i\in I_{h}}\int_{\Omega}|\nabla v^{\varepsilon}_{i}|^{2}+\lambda_{i}| v^{\varepsilon}_{i}|^{2} \, dx< 4l_{h},~~ \forall h=p+1, \ldots, m,
\end{align}
for small $\varepsilon$. Note that
\begin{equation}\label{3-p-4.4}
\|\mathbf{u}_{k}\|_{k}^{2}\leq \sum_{h=1}^p\|\mathbf{u}_{h}\|_{h}^{2}=4c_{1,\ldots,p}\leq 4\overline{C}.
\end{equation}
We deduce from \eqref{3-p-4.2}, \eqref{3-p-4.3} and \eqref{3-p-4.4} that
\begin{align}\label{3-p-4-1}
\Phi(t_{1},\cdots,t_{m})&\leq  \sum_{k=1}^p 2t_k\overline C + \sum_{h=p+1}^m 2t_k l_k-  \sum_{k=1}^{m}\frac{1}{4}\theta t_k^2\\
				& \leq \sum_{k=1}^m \left(2t_k \max\{\overline C,l_{1},\ldots, l_m\}-\frac{1}{4}\theta t_k^2\right),
\end{align}
which yields that $\Phi(t_{1},\cdots,t_{m})<0$ when $t_{1},\ldots,t_{m}>\widehat{t}$, where
\begin{equation}\label{1,2,3-7-3}
\widehat{t}= \frac{8\max\{\overline{C},l_1,\ldots, l_m \}}{\theta}.
\end{equation}
Thus,
\begin{align}\label{1,2,3-7}
\max_{t_{1},\ldots,t_{m}>0}& J\left(\sqrt{t_{1}}\mathbf{u}_{1},\cdots, \sqrt{t_{p}}\mathbf{u}_{p},\sqrt{t_{p+1}}\mathbf{v}_{p+1}^{\varepsilon},\cdots,\sqrt{t_{m}}\mathbf{v}_{m}^{\varepsilon}\right)
 =\nonumber\\\
& \max_{0<t_{1},\ldots,t_{m} \leq \widehat{t}}J\left(\sqrt{t_{1}}\mathbf{u}_{1},\cdots, \sqrt{t_{p}}\mathbf{u}_{p},\sqrt{t_{p+1}}\mathbf{v}_{p+1}^{\varepsilon},\cdots,\sqrt{t_{m}}\mathbf{v}_{m}^{\varepsilon}\right).
\end{align}
By \eqref{3-p-23} and \eqref{3-p-16} we have, for $k=1,\ldots, p$, $h=p+1,\ldots,m$ and   $t_{h}>0$, $0<t_{k}<\widehat{t}$,
\begin{align}\label{1,2-11-m}
\sum_{k=1}^{p}t_{k}t_{h}\left|\sum_{(i,j)\in I_{k}\times I_{h}}\beta_{ij}
\int_{\Omega}|u_{i}|^{2}|v^{\varepsilon}_{j}|^{2}\right| &\leq d\max_{(i,j)\in \mathcal{K}_{2}}|\beta_{ij}|\widehat{t}\Pi^{2}t_{h}\sum_{i\in I_{h}}\int_{B_{2\rho}(x_{0})}|v^{\varepsilon}_{i}|^{2}\nonumber\\
& \leq -\frac{1}{2}t_{h}\sum_{i\in I_{h}}\lambda_{i}\int_{\Omega}|v^{\varepsilon}_{i}|^{2}
\end{align}
(recall that $\lambda_i<0$).
Hence, since $v_i^\varepsilon\cdot v_j^\varepsilon=0$ whenever $i\in I_h,j\in I_l$, $h\neq l$, $h,l\in \{p+1,\ldots, m\}$, we have that
\begin{align}\label{1,2,3-8}
J(&\sqrt{t_{1}}\mathbf{u}_{1},\cdots, \sqrt{t_{p}}\mathbf{u}_{p},\sqrt{t_{p+1}}\mathbf{v}_{p+1}^{\varepsilon},\cdots,\sqrt{t_{m}}\mathbf{v}_{m}^{\varepsilon}) =\frac{1}{2}\sum_{k=1}^{p}t_{k}\|\mathbf{u}_{k}\|_{k}^{2}
-\frac{1}{4}\sum_{k,l=1}^{p}t_{k}t_{l}\sum_{(i,j)\in I_{k}\times I_{l}}\int_{\Omega}\beta_{ij} u_{i}^{2}u_{j}^{2} \nonumber\\
&+\sum_{h=p+1}^{m}\left(\frac{t_{h}}{2}\|\mathbf{v}^{\varepsilon}_{h}\|_{h}^{2}-\frac{t_{h}^{2}}{4}\sum_{(i,j)\in I_{h}^{2}}\int_{\Omega}\beta_{ij}| v^{\varepsilon}_{i}|^{2}|v^{\varepsilon}_{j}|^{2}\right)
-\sum_{k=1}^{p}\sum_{h=p+1}^{m}\frac{t_{k}t_{h}}{2}\sum_{(i,j)\in I_{k}\times I_{h}}\int_{\Omega}\beta_{ij} u_{i}^{2}|v^{\varepsilon}_{j}|^{2}\nonumber\\
& \leq\frac{1}{2}\sum_{k=1}^{p}t_{k}\|\mathbf{u}_{k}\|_{k}^{2}
-\frac{1}{4}\sum_{k,l=1}^{p}t_{k}t_{l}\sum_{(i,j)\in I_{k}\times I_{l}}\int_{\Omega}\beta_{ij} u_{i}^{2}u_{j}^{2} \nonumber\\
& \quad +\sum_{h=p+1}^{m}\left(\frac{1}{2}t_{h}\sum_{i\in I_{h}}\int_{\Omega}|\nabla v^{\varepsilon}_{i}|^{2}+\frac{\lambda_{i}}{2}| v^{\varepsilon}_{i}|^{2}\, dx-\frac{1}{4}t_{h}^{2}\sum_{(i,j)\in I_{h}^{2}}\int_{\Omega}\beta_{ij}| v^{\varepsilon}_{i}|^{2}|v^{\varepsilon}_{j}|^{2}\right)\nonumber\\
& =: f(t_{1},\ldots,t_{p})+ g_\varepsilon(t_{p+1},\ldots,t_{m}).
\end{align}

We claim that
\begin{equation}\label{3-p-2}
\max_{t_{1},\cdots,t_{p}>0}f(t_{1},\cdots,t_{p})=f(1,\cdots,1)=J_{1,\cdots,p}(\mathbf{u}^{p})=c_{1,\cdots,p}.
\end{equation}
It follows from Lemma \ref{Positive Definite} that $M_{B}^{1,\cdots,p}(\mathbf{u}^{p})$ is positive definite. Then we see that $f$ is strictly concave and has a maximum point in $\overline{\mathbb{R}_{+}^p}$. Note that the point $\mathbf{1}=(1,\ldots,1)$ is a critical point of $f$ due to the fact that $\mathbf{u}^{p}\in \mathcal{N}_{1,\ldots,p}$. Moreover, it is easy to see that $f$ is of class $C^1$ in $\overline{\mathbb{R}_{+}^p}$. Therefore, we know that $\mathbf{1}$ is the unique critical point of $f$ and is a global maximum by strict concavity, which yields that \eqref{3-p-2} holds.

On the other hand, similarly to \eqref{2.15} we see that
\begin{equation}\label{3-p-3}
\max_{t_{p+1},\cdots,t_{m}>0}g_\varepsilon (t_{p+1},\cdots,t_{m})<\sum_{h=p+1}^{m}l_{h}-\delta,
\end{equation}
for the same $\delta$ as in \eqref{4-p-9}, by taking $\varepsilon$ sufficiently small.
It follows from \eqref{1,2,3-7}, \eqref{1,2,3-8}, \eqref{3-p-2} and \eqref{3-p-3} that
\begin{align}\label{3-p-3-6}
\max_{t_{1},\ldots,t_{m}>0}& J\left(\sqrt{t_{1}}\mathbf{u}_{1},\cdots, \sqrt{t_{p}}\mathbf{u}_{p},\sqrt{t_{p+1}}\mathbf{v}_{p+1}^{\varepsilon},\cdots,\sqrt{t_{m}}\mathbf{v}_{m}^{\varepsilon}\right)
 =\nonumber\\
& \max_{0<t_{1},\ldots,t_{m} \leq \widehat{t}}J\left(\sqrt{t_{1}}\mathbf{u}_{1},\cdots, \sqrt{t_{p}}\mathbf{u}_{p},\sqrt{t_{p+1}}\mathbf{v}_{p+1}^{\varepsilon},\cdots,\sqrt{t_{m}}\mathbf{v}_{m}^{\varepsilon}\right)\nonumber\\
& \leq \max_{t_{1},\cdots,t_{p}>0}f(t_{1},\cdots,t_{p})+\max_{t_{p+1},\cdots,t_{m}>0}g(t_{p+1},\cdots,t_{m})\nonumber\\
& <c_{1,\ldots,p}+\sum_{h=p+1}^{m}l_{h}-\delta.  \qedhere
\end{align}
\end{proof}

By the previous lemma, we know that $\Phi$ has a global maximum $\mathbf{t}^{\varepsilon}$ in $\overline{\mathbb{R}^{m}_{+}}$ for sufficiently small $\varepsilon>0$. From now on we fix such an $\varepsilon$. It is easy to see that $\partial_{k}\Phi(\mathbf{t}^{\varepsilon})\leq 0$ if $t^{\varepsilon}_{k}=0$ and $\partial_{k}\Phi(\mathbf{t}^{\varepsilon})= 0$ if $t^{\varepsilon}_{k}>0$. Moreover, for the latter case we have
\begin{equation}\label{3-p-12}
\|\mathbf{u}^{\varepsilon}_{k}\|_{k}^{2}=\sum_{h=1}^{m}M_{B}(\mathbf{u}^{\varepsilon})_{kh}t^{\varepsilon}_{h}, ~\text{ for } t^{\varepsilon}_{k}>0.
\end{equation}

Recall the definition of $\Lambda_{3}$ (see \eqref{Constant-4-1}). In the following we show that \eqref{3-p-3-4} holds.

\begin{lemma}\label{t-positive}
If
\begin{equation*}
\beta_{ij}\geq 0  \quad\forall (i,j)\in \mathcal{K}_{1} ~and~
-\infty<\beta_{ij}< \Lambda_{3} \quad\forall (i,j)\in \mathcal{K}_{2},
\end{equation*}
then there holds
\begin{equation*}
t^{\varepsilon}_{1},\cdots,t^{\varepsilon}_{k},\cdots, t^{\varepsilon}_{m}>0.
\end{equation*}
\end{lemma}
\begin{proof}[\bf{Proof}]
Firstly, we claim that
\begin{equation}\label{3-p-13}
\sum_{k=1}^{m}t^{\varepsilon}_{k}\|\mathbf{u}^{\varepsilon}_{k}\|_{k}^{2}\leq 4\left(\sum_{k=1}^{m}l_{k}-2\delta\right).
\end{equation}
We deduce from \eqref{3-p-12} that for $t^{\varepsilon}_{k}>0$ we have
\begin{equation}\label{3-p-12-2}
\|\mathbf{u}^{\varepsilon}_{k}\|_{k}^{2}t^{\varepsilon}_{k}=\sum_{h=1}^{m}M_{B}(\mathbf{u}^{\varepsilon})_{kh}t^{\varepsilon}_{h}t^{\varepsilon}_{k}, \end{equation}
and clearly \eqref{3-p-12-2} also holds for $t^{\varepsilon}_{k}=0$. Hence, we know that \eqref{3-p-12-2} is true for $k=1,\ldots, m$. It follows that
\begin{equation}\label{3-p-14-2}
J(\sqrt{t_{1}^{\varepsilon}}\mathbf{u}_{1},\cdots,\sqrt{t_{p}^{\varepsilon}}\mathbf{u}_{p},\cdots,\sqrt{t_{m}^{\varepsilon}}\mathbf{v}_{m}^{\varepsilon})
= \frac{1}{4}\sum_{k=1}^{m}\|\mathbf{u}^{\varepsilon}_{k}\|_{k}^{2}t^{\varepsilon}_{k}.
\end{equation}
Note that $c_{1,\ldots,p}<\sum_{k=1}^{p}l_{k}-\delta$ by Theorem \ref{Energy Estimates}. Combining this with \eqref{eq:maximum_estimate} we see that
\begin{equation}\label{3-p-14}
J(\sqrt{t_{1}^{\varepsilon}}\mathbf{u}_{1},\cdots,\sqrt{t_{p}^{\varepsilon}}\mathbf{u}_{p},\cdots,\sqrt{t_{m}^{\varepsilon}}\mathbf{v}_{m}^{\varepsilon})
\leq \sup_{t_{1},\ldots,t_{m}\geq 0}J(\sqrt{t_{1}}\mathbf{u}_{1},\cdots,\sqrt{t_{p}}\mathbf{u}_{p},\cdots,\sqrt{t_{m}}\mathbf{v}_{m}^{\varepsilon})\leq\sum_{k=1}^{m}l_{k}-2\delta.
\end{equation}
We infer from \eqref{3-p-14-2} and \eqref{3-p-14} that \eqref{3-p-13} holds.

By contradiction, we suppose that $t^{\varepsilon}_{1}=0$, and the maximum point has the form $(0, t_{2}^{\varepsilon},\cdots,t_{m}^{\varepsilon})$. Consider the function
\begin{align}\label{3-p-19}
\varphi(t_1):=J(\sqrt{t_{1}}\mathbf{u}_{1}^{\varepsilon},\cdots,\sqrt{t_{m}^{\varepsilon}}\mathbf{u}_{m}^{\varepsilon})   &=\frac{1}{2}t_{1}\|\mathbf{u}_{1}^{\varepsilon}\|_{1}^{2}-\frac{1}{2}\sum_{k=2}^{m}M_{B}(\mathbf{u}^{\varepsilon})_{1k}t_{1}t^{\varepsilon}_{k}-\frac{1}{4}M_{B}(\mathbf{u}^{\varepsilon})_{11}t_{1}^{2}\nonumber\\ & \quad + \frac{1}{2}\sum_{k=2}^{m}t_{k}^{\varepsilon}\|\mathbf{u}_{k}^{\varepsilon}\|_{k}^{2}-\frac{1}{4}\sum_{k,h=2}^{m}M_{B}(\mathbf{u}^{\varepsilon})_{kh}t_{k}^{\varepsilon}t_{h}^{\varepsilon}.
\end{align}
By assumption we have that
\begin{align*}
\sum_{k=2}^{m}M_{B}(\mathbf{u}^{\varepsilon})_{1k}t_{k}^{\varepsilon}&=\sum_{k=2}^{m}t_{k}^{\varepsilon}\sum_{(i,j)\in I_{1}\times I_{k}}\int_{\Omega}\beta_{ij}| u^{\varepsilon}_{i}|^{2}|u^{\varepsilon}_{j}|^{2}\leq \frac{\Lambda_{2}}{S^{2}}\sum_{k=2}^{m}t_{k}^{\varepsilon}\sum_{(i,j)\in I_{1}\times I_{k}}\|u^{\varepsilon}_{i}\|_{i}^{2}\|u^{\varepsilon}_{j}\|_{j}^{2}\\
& \leq \frac{\Lambda_{2}}{S^{2}}\|\mathbf{u}_{1}^{\varepsilon}\|_{1}^{2}\sum_{k=2}^{m}t_{k}^{\varepsilon}\|\mathbf{u}^{\varepsilon}_{k}\|_{k}^{2}
\leq \frac{4\Lambda_{2}\left(\sum_{k=1}^{m}l_{k}-2\delta\right)}{S^{2}}\|\mathbf{u}_{1}^{\varepsilon}\|_{1}^{2}\leq \frac{1}{4}\|\mathbf{u}_{1}^{\varepsilon}\|_{1}^{2}.
\end{align*}
Combining this with \eqref{3-p-19} we have
\begin{align}\label{3-p-20}
\frac{1}{2}&t_{1}\|\mathbf{u}_{1}^{\varepsilon}\|_{1}^{2}-\frac{1}{2}\sum_{k=2}^{m}M_{B}(\mathbf{u}^{\varepsilon})_{1k}t_{1}t^{\varepsilon}_{k}-\frac{1}{4}M_{B}(\mathbf{u}^{\varepsilon})_{11}t_{1}^{2}\nonumber\\   &=\frac{1}{2}t_{1}\left(\|\mathbf{u}_{1}^{\varepsilon}\|_{1}^{2}-\sum_{k=2}^{m}M_{B}(\mathbf{u}^{\varepsilon})_{1k}t^{\varepsilon}_{k}-\frac{1}{2}M_{B}(\mathbf{u}^{\varepsilon})_{11}t_{1}\right)
 \geq\frac{1}{2}t_{1}\left(\frac{3}{4}\|\mathbf{u}_{1}^{\varepsilon}\|_{1}^{2}- \frac{1}{2}M_{B}(\mathbf{u}^{\varepsilon})_{11}t_{1}\right)>0,
\end{align}
for $t_{1}>0$ small enough. It follows from \eqref{3-p-19} and \eqref{3-p-20} that
\begin{equation*}
J\left(\sqrt{t_{1}}\mathbf{u}_{1}^{\varepsilon},\sqrt{t_{2}^{\varepsilon}}\mathbf{u}_{2}^{\varepsilon},\cdots,\sqrt{t_{m}^{\varepsilon}}\mathbf{u}_{m}^{\varepsilon}\right)
>J\left(0,\sqrt{t_{2}^{\varepsilon}}\mathbf{u}_{2}^{\varepsilon},\cdots,\sqrt{t_{m}^{\varepsilon}}\mathbf{u}_{m}^{\varepsilon}\right),
\end{equation*}
which is a contradiction.
\end{proof}

\begin{proof}[\bf Proof of Theorem \ref{Energy Estimation-1,2,m-1}.]
Without loss of generality, we prove that $c<c_{1,\ldots,p}+\sum_{h=p+1}^{m}l_{h}-\delta$.
By Lemma \ref{t-positive}, there exists $t_{1}^{\varepsilon},\ldots,t_{m}^{\varepsilon}>0$ such that
$$
\left(\sqrt{t_{1}^{\varepsilon}}\mathbf{u}_{1},\cdots,\sqrt{t_{p}^{\varepsilon}}\mathbf{u}_{p}, \sqrt{t_{p+1}^{\varepsilon}}\mathbf{v}_{p+1}^{\varepsilon},\cdots,\sqrt{t_{m}^{\varepsilon}}\mathbf{v}_{m}^{\varepsilon}\right)\in \mathcal{N}.
$$
It follows from Lemma \ref{Con-4-2} that
\begin{align*}
c &\leq J\left(\sqrt{t_{1}^{\varepsilon}}\mathbf{u}_{1},\cdots,\sqrt{t_{p}^{\varepsilon}}\mathbf{u}_{p},\sqrt{t_{p+1}^{\varepsilon}}\mathbf{v}_{p+1}^{\varepsilon},\cdots,\sqrt{t_{m}^{\varepsilon}}\mathbf{v}_{m}^{\varepsilon}\right)\\
&\leq
\max_{t_{1},\ldots,t_{m}>0}J(\sqrt{t_{1}}\mathbf{u}_{1},\cdots,\sqrt{t_{p}^{\varepsilon}}\mathbf{u}_{p},\sqrt{t_{p+1}}\mathbf{v}_{p+1}^{\varepsilon},\cdots,\sqrt{t_{m}}\mathbf{v}_{m}^{\varepsilon}) <c_{1,\ldots,p}+\sum_{h=p+1}^{m}l_{h}-\delta. \qedhere
\end{align*}
\end{proof}

\begin{remark}\label{5.1}
We mention that Theorem \ref{Energy Estimation-1,2,m-1} also holds for $d=m=2$, where in this case it gives the estimate:
\begin{equation}\label{eq:accurateestimate}
c<\min\{c_1+l_2-\delta,c_2+l_1-\delta\},
\end{equation}
where $\delta$ is independent of $\beta:=\beta_{12}=\beta_{21}$. Observe that $l_{i}=\frac{1}{4}\beta_{ii}^{-1}\widetilde{S}^{2}, i=1,2$, where $\widetilde{S}$ is the Sobolev best constant. This accurate estimate allows us to answer a question which was left open in \cite{Zou 2012}, where the authors dealt with the case $d=m=2$. Indeed, let us recall the following statement from that paper:
\begin{thm}\cite[Theorem 1.4]{Zou 2012}
Assume that $-\lambda_1(\Omega)<\lambda_1\leq \lambda_2<0$. Let $\beta_{n}<0, n\in \mathbb{N}$ satisfy $\beta_{n}\rightarrow -\infty$ as $n\rightarrow \infty$, and let $(u_n,v_n)$ be a least energy positive solution of \eqref{S-system} with $\beta_{12}=\beta_{21}=\beta_{n}$. Then $\int_{\Omega}\beta_{n}u_n^2v_n^2\rightarrow 0$ as $n\rightarrow \infty$ and, passing to a subsequence, one of the following conclusions holds.
\begin{itemize}
  \item [(1)] $u_n\rightarrow u_{\infty}$ strongly in $H_0^1(\Omega)$ and $v_n\rightharpoonup 0$ weakly in $H_0^1(\Omega)$, where $u_{\infty}$ is a least energy positive solution of
      \begin{equation}
      -\Delta u+\lambda_{1}u=\beta_{11}u^3, \qquad u\in H_0^1(\Omega).
      \end{equation}
  \item [(2)] $v_n\rightarrow v_{\infty}$ strongly in $H_0^1(\Omega)$ and $u_n\rightharpoonup 0$ weakly in $H_0^1(\Omega)$, where $v_{\infty}$ is a least energy positive solution of
      \begin{equation}
      -\Delta v+\lambda_{2}v=\beta_{22}v^3, \qquad v\in H_0^1(\Omega).
      \end{equation}
  \item [(3)] $(u_n,v_n)\rightarrow (u_{\infty},v_{\infty})$ strongly in $H_0^1(\Omega)\times H_0^1(\Omega)$ and $u_{\infty}\cdot v_{\infty}\equiv 0$, where $u_{\infty}\in C(\overline{\Omega})$ is a least energy positive solution of
      \begin{equation}
      -\Delta u+\lambda_{1}u=\beta_{11}u^3, \qquad u\in H_0^1(\{u_{\infty}>0\}),
      \end{equation}
      and $v_{\infty}\in C(\overline{\Omega})$ is a least energy positive solution of
      \begin{equation}
      -\Delta v+\lambda_{2}v=\beta_{22}v^3, \qquad v\in H_0^1(\{v_{\infty}>0\}).
      \end{equation}
      Furthermore, both $\{u_{\infty}>0\}$ and $\{v_{\infty}>0\}$ are connected domains, and $\{v_{\infty}>0\}=\Omega\backslash\overline{\{u_{\infty}>0\}}$.
\end{itemize}
\end{thm}

The question left open in \cite{Zou 2012} was whether $(1)$ and $(2)$ could actually happen. Combining \cite[Remark 6.1]{Zou 2012} with \eqref{eq:accurateestimate} we can show that actually $(1)$ and $(2)$ \emph{cannot} happen, therefore $(3)$ always holds.

\end{remark}

\subsection{Construction of Palais-Smale sequence}\label{subsec3.3}

In this subsection, we construct a Palais-Smale sequence at level $c_\Gamma$.
Recalling that $\Lambda_{1}=S^{2}/(32\overline{C})$ (see \eqref{Constant-6}) we have the following result.
\begin{prop}\label{PS Sequence}
Assume that
\begin{equation*}
\beta_{ij}\geq 0  \quad\forall (i,j)\in \mathcal{K}_{1} ~and~
-\infty<\beta_{ij}< \Lambda_{1} \quad\forall (i,j)\in \mathcal{K}_{2}.
\end{equation*}
Then there exists a sequence $\{\mathbf{u}_{n}\}\subset\mathcal{N}_\Gamma$ satisfying
\begin{equation}\label{Inequality16-1}
\lim_{n\rightarrow\infty}J_\Gamma(\mathbf{u}_{n})=c_\Gamma,\quad \lim_{n\rightarrow\infty}J_\Gamma'(\mathbf{u}_{n})=0.
\end{equation}
\end{prop}
\begin{proof}[\bf{Proof}]
We adapt the idea of the proof \cite[Lemma 2.5]{Tavares 2016-1} to the critical case.
Notice that $J_\Gamma$ is coercive and bounded from below on $\mathcal{N}_\Gamma$. By Lemma \ref{Pre-1} we can assume that $J_\Gamma(\mathbf{u}_{n})\leq 2\overline{C}$ for $n$ large enough. Then by the Ekeland variational principle (which we can use, since $\mathcal{N}_\Gamma\cap \left\{\mathbf{u}: J_\Gamma(\mathbf{u})\leq 2\overline{C}\right\}$ is a closed set by Lemma \ref{Bounded}), there exists a minimizing sequence $\{\mathbf{u}_{n}\}\subset\mathcal{N}_\Gamma$ satisfying
\begin{equation}\label{Inequality14}
J_\Gamma(\mathbf{u}_{n})\rightarrow c_\Gamma,\quad
J_\Gamma'(\mathbf{u}_{n})-\sum_{k\in \Gamma}\lambda_{k,n}\Psi'_{k}(\mathbf{u}_{n})=o(1), \text{ as } n\rightarrow \infty,
\end{equation}
where
\begin{equation}\label{Func-1}
\Psi_{k}(\mathbf{u}):= \|\mathbf{u}_{k}\|_{k}^{2}-\sum_{h\in \Gamma}\sum_{(i,j)\in I_{k}\times I_{h}}\int_{\Omega}\beta_{ij}u^{2}_{i}u^{2}_{j}.
\end{equation}
Since $J_\Gamma(\mathbf{u}_{n})\leq 2\overline{C}$ for $n$ large enough, $\{u^{n}_i\}$ is uniformly bounded in $H^1_0(\Omega)$. We suppose that, up to a subsequence,
\begin{equation*}
u_{i}^{n}\rightarrow u_{i} \text{ weakly in } H^{1}_{0}(\Omega),
\end{equation*}
and
\begin{equation}\label{limit-1}
\lim_{n\rightarrow\infty}\sum_{(i,j)\in I_{h}\times I_{k}}\int_{\Omega}\beta_{ij}| u_{i}^{n}|^{2}|u_{j}^{n}|^{2}:= a_{hk}, \quad \widetilde{B}:= (a_{hk})_{ h,k\in \Gamma}.
\end{equation}
Let us prove that $\widetilde{B}$ is positive definite.

We claim that
\begin{align}\label{Matrix-2-7-1}
\sum_{(i,j)\in I_{k}^{2}}\int_{\Omega}\beta_{ij}| u_{i}^{n}|^{2}|u_{j}^{n}|^{2} -\sum_{h\neq k}\left|\sum_{(i,j)\in I_{k}\times I_{h}}\int_{\Omega}\beta_{ij}| u_{i}^{n}|^{2}|u_{j}^{n}|^{2}\right|\geq \frac{S}{2}C_{1},
\end{align}
where $C_{1}$ is defined in Lemma \ref{Bounded}. We can prove this claim by performing exactly as in the proof of \eqref{3-p-7}, due to the fact that $\mathbf{u}_{n}\subset\mathcal{N}_\Gamma$ and $\sum_{k\in\Gamma}\|\mathbf{u}_{k}^{n}\|_{k}^{2}\leq 8\overline{C}$.
Then we see from \eqref{Matrix-2-7-1} that
\begin{equation}\label{Matrix-3-1}
a_{kk}-\sum_{h\neq k}|a_{hk}|\geq \frac{S}{2}C_{1}>0  ~~k\in \Gamma,
\end{equation}
which yields that $\widetilde{B}$ is positive definite.

Consider $\overline{\mathbf{u}}^k_n$ defined by
\begin{equation}
\overline{u}^k_{i,n}:=
\begin{cases}
u_{i}^n ~~\text{ if } i\in I_k,\\
0 ~~\text{ if } i\not\in I_k.
\end{cases}
\end{equation}
Testing the second equation in \eqref{Inequality14}, together with \eqref{limit-1} we get that
\begin{equation}\label{limit-2}
\mathbf{o}(1)=M_B^\Gamma(\mathbf{u}_n)\mathbf{\lambda}^\Gamma_n=(\widetilde{B}+\mathbf{o}(1))\mathbf{\lambda}^\Gamma_n,
\end{equation}
where $\mathbf{\lambda}^\Gamma_n:=(\lambda_{k,n})_{ |\Gamma|\times 1}$, $k\in \Gamma$. Denote by $(\mathbf{\lambda}^\Gamma_n)'$ the transpose of $\mathbf{\lambda}^\Gamma_n$. Multiplying the equation \eqref{limit-2} by $(\mathbf{\lambda}^\Gamma_n)'$, we deduce from the positive definiteness of $\widetilde{B}$ that
\begin{equation}
\mathbf{o}(1)|(\mathbf{\lambda}^\Gamma_n)'|\geq C|(\mathbf{\lambda}^\Gamma_n)'|^2+\mathbf{o}(1)|(\mathbf{\lambda}^\Gamma_n)'|^2,
\end{equation}
where $C$ is independent on $n$. It follows that $\lambda_{k,n}\rightarrow 0$ as $n\rightarrow \infty$. Observe that ${\Psi_{k}'(\mathbf{u}_n)}$ is a uniformly bounded family of operators and that $\{u^{n}_i\}$ is uniformly bounded in $H^1_0(\Omega)$; then by \eqref{Inequality14} we get that $J_\Gamma'(\mathbf{u}_{n})\rightarrow 0$. Therefore, $\mathbf{u}_{n}$ is a standard Palais-Smale sequence.
\end{proof}

\section{Proof of the main theorems}\label{sec4}
As stated in the introduction, we prove Theorem \ref{Theorem-1} by induction in the number of sub-groups. We start in the next subsection by proving the first induction step.
\subsection{Proof of Theorem \ref{Theorem-1} for the case of one group} \label{subsec4.1}
Given $h=1,\ldots, m$, consider the following system
\begin{equation}\label{System-h-1}
\begin{cases}
-\Delta u_{i}+\lambda_{i}u_{i}=\sum_{j \in I_{h}}\beta_{ij}u_{j}^{2}u_{i}  \text{ in } \Omega,\\
u_{i}\in H^{1}_{0}(\Omega)  \quad \forall i\in I_{h}.
\end{cases}
\end{equation}
We will consider the levels $c_\Gamma$ with $|\Gamma|=1$. We recall that, with $\Gamma=\{h\}$ (and dropping the parenthesis in the notations from now on):
\begin{equation*}
J_{h}(\mathbf{u}):=\int_{\Omega}\frac{1}{2}\sum_{i\in I_{h}}(|\nabla u_{i}|^{2}+\lambda_{i}u_{i}^{2})-\frac{1}{4}\sum_{(i,j)\in I_{h}^{2}}
\beta_{ij}u_{j}^{2}u_{i}^{2}\, dx,
\end{equation*}
and
\begin{equation}\label{Energy-h}
c_{h}:= \inf_{\mathbf{u}\in \mathcal{N}_{h}}J_{h}(\mathbf{u})>0,
\end{equation}
where
\begin{equation*}
\mathcal{N}_{h}:=\Big\{\mathbf{u}\in (H^{1}_{0}(\Omega))^{a_{h}-a_{h-1}}:\mathbf{u}\neq \mathbf{0} \text{ and }
\langle \nabla J_{h}(\mathbf{u}),\mathbf{u}\rangle=0 \Big\}.
\end{equation*}
From Theorem \ref{Energy Estimates} we know that the following lemma holds.
\begin{lemma}\label{Energy Estimation-h}
Assume that $\beta_{ij}\geq0$ for every $(i,j)\in I_{h}^{2}, i\neq j$, $\beta_{ii}>0, i\in I_{h}$. Then
\begin{equation*}
c_{h}<l_{h}-\delta.
\end{equation*}
\end{lemma}

The following lemma is the counterpart of Br\'ezis-Lieb Lemma (see \cite{Brezis Lieb lemma}) for two component and its proof can be found in \cite[p. 538]{Zou 2015}.
\begin{lemma}\label{BL}
Let $u_{n}\rightharpoonup u, v_{n}\rightharpoonup v$ in $H^{1}_{0}(\Omega)$ as $n\rightarrow \infty$. Then, up to a subsequence, there holds
\begin{equation*}
\lim_{n\rightarrow \infty}\int_{\Omega}(|u_{n}|^{2}|v_{n}|^{2}-|u_{n}-u|^{2}|v_{n}-v|^{2}-|u|^{2}|v|^{2})=0.
\end{equation*}
\end{lemma}

\begin{thm}\label{System-h}
Assume that $\beta_{ij}\geq0$ for every $(i,j)\in I_{h}^{2}, i\neq j$, $\beta_{ii}>0, i\in I_{h}$. Then
 $c_{h}$ is achieved by a nonnegative $\mathbf{u}\in \mathcal{N}_{h}$. Moreover, any minimizer is a nonnegative solution of \eqref{System-h-1}.
\end{thm}
\begin{proof}[\bf{Proof}]
From Proposition \ref{PS Sequence} there exists $\{\mathbf{u}_{n}\}$ such that
\begin{equation*}
\lim_{n\rightarrow \infty}J_{h}(\mathbf{u}_{n})=c_{h},\quad \lim_{n\rightarrow \infty}J'_{h}(\mathbf{u}_{n})=0.
\end{equation*}
It is standard to see that $\{u^{n}_{i}\}$ is bounded in $H^{1}_{0}(\Omega)$, and so we may assume that $u^{n}_{i}\rightharpoonup u_{i}$ weakly in $H^{1}_{0}(\Omega)$, $i\in I_{h}$. Passing to a subsequence, we may assume that
\begin{equation*}
u_{i}^{n}\rightharpoonup u_{i}, ~weakly ~in ~L^{4}(\Omega), ~i\in I_{h}.
\end{equation*}
\begin{equation*}
(u_{i}^{n})^{2}\rightharpoonup u_{i}^{2}, ~weakly ~in ~L^{2}(\Omega), ~i\in I_{h}.
\end{equation*}
\begin{equation}\label{h-1}
u_{i}^{n}\rightarrow u_{i}, ~strongly ~in ~L^{2}(\Omega), ~i\in I_{h}.
\end{equation}
Set $v_{i}^{n}=u_{i}^{n}-u_{i}$, and so
\begin{equation}\label{h-2-1}
v_{i}^{n}\rightharpoonup 0 \text{ weakly } \text{ in } H^{1}_{0}(\Omega), \quad v_{i}^{n}\rightarrow 0 \text{ strongly in }  L^{2}(\Omega).
\end{equation}
Note that since $\mathbf{u}_{n}\in \mathcal{N}_{h}$, then we have
\begin{equation}\label{h-3-2}
\sum_{i\in I_{h}}\int_{\Omega}|\nabla u_{i}^{n}|^{2}+\lambda_{i}|u_{i}^{n}|^{2}\, dx=
\sum_{(i,j)\in I_{h}^{2}}\int_{\Omega}\beta_{ij}| u_{i}^{n}|^{2}|u_{j}^{n}|^{2}.
\end{equation}
We deduce from \eqref{h-1} and \eqref{h-2-1} that
\begin{equation}\label{h-2-2}
\int_{\Omega}|\nabla u_{i}^{n}|^{2}=\int_{\Omega}|\nabla v_{i}^{n}|^{2}+|\nabla u_{i}|^{2}+o(1), \quad \int_{\Omega}|u_{i}^{n}|^{2}=\int_{\Omega}u_{i}^{2}+o(1),
\end{equation}
and by Lemma \ref{BL} we have
\begin{equation}\label{h-2}
\int_{\Omega}| u_{i}^{n}|^{2}|u_{j}^{n}|^{2}=\int_{\Omega}| v_{i}^{n}|^{2}|v_{j}^{n}|^{2}+ \int_{\Omega}| u_{i}|^{2}|u_{j}|^{2}+o(1).
\end{equation}
It follows from \eqref{h-3-2}-\eqref{h-2} and $J_{h}'(\mathbf{u}_{h})\mathbf{u}_{h}=0$ that
\begin{equation}\label{h-3}
\sum_{i\in I_{h}}\int_{\Omega}|\nabla v_{i}^{n}|^{2}-
\sum_{(i,j)\in I_{h}^{2}}\int_{\Omega}\beta_{ij}| v_{i}^{n}|^{2}|v_{j}^{n}|^{2}=o(1).
\end{equation}
We deduce from \eqref{h-2-2}-\eqref{h-3} that
\begin{equation}\label{thm-4-2-0}
J_h(\mathbf{u}_{n})=J_h(\mathbf{u}_h)+\frac{1}{4}\sum_{i\in I_{h}}\int_{\Omega}|\nabla v_{i}^{n}|^{2}+o(1).
\end{equation}
Passing to a subsequence, we may assume that
\begin{equation*}
\lim_{n\rightarrow\infty}\sum_{i\in I_{h}}\int_{\Omega}|\nabla v_{i}^{n}|^{2}=d_{h}.
\end{equation*}
Combining this with \eqref{thm-4-2-0} yields
\begin{equation}\label{S-h-1}
0\leq J_{h}(\mathbf{u}_{h})\leq J_{h}(\mathbf{u}_{h})+\frac{1}{4}d_{h}=\lim_{n\rightarrow \infty}J_{h}(\mathbf{u}_{n})=c_{h}.
\end{equation}

Next, we prove that $\mathbf{u}_{h}\neq \mathbf{0}$. By contradiction, assume that $\mathbf{u}_{h}\equiv \mathbf{0}$. By \eqref{S-h-1} we have $d_{h}=4c_h>0$. Then by \eqref{Vector Sobolev Inequality} and \eqref{h-3} we get that
\begin{align*}
\sum_{i\in I_{h}}\int_{\Omega}|\nabla v_{i}^{n}|^{2} =
\sum_{(i,j)\in I^{2}_{h}}\int_{\Omega}\beta_{ij}| v_{i}^{n}|^{2}|v_{j}^{n}|^{2}+o(1)
   \leq (4l_{h})^{-1}\left(\sum_{i\in I_{h}}\int_{\Omega}|\nabla v_{i}^{n}|^{2}\right)^{2}+o(1),
\end{align*}
which yields that $d_{h}\geq 4l_{h}$. It follows from \eqref{S-h-1} that
\begin{equation*}
c_{h}=\frac{1}{4}d_{h}\geq l_{h},
\end{equation*}
which contradicts Lemma \ref{Energy Estimation-h}. Therefore, $\mathbf{u}_{h}\neq\mathbf{0}$ and $\mathbf{u}_{h}\in \mathcal{N}_{h}$. Then we see from \eqref{S-h-1} that
\begin{equation*}
c_{h}\leq J_{h}(\mathbf{u}_{h})\leq\lim_{n\rightarrow \infty}J_{h}(\mathbf{u}_{n})=c_{h}.
\end{equation*}
That is, $J_{h}(\mathbf{u}_{h})=c_{h}$, and so $J_{h}(|\mathbf{u}_{h}|)=c_{h}$. By Lemma \ref{Lagrange} we get that $|\mathbf{u}_{h}|$ is a nonnegative solution of \eqref{System-h-1}.
\end{proof}

\subsection{Proof of Theorem \ref{Theorem-1} in the general case}\label{subsec4.2}
Set
$$
\Lambda_{4}:= \frac{\delta S^{2}}{8\overline{C}\sum_{l=1}^{m}l_{h}},
$$
where $l_{h}$ is defined in \eqref{eq:Ground State-2}, $\overline{C}$ is defined in \eqref{Pre-2}, and $\delta$ is defined in Theorem \ref{Energy Estimates}.
Let
\begin{equation}\label{Constant-8}
\Lambda:= \min\{\Lambda_{3},\Lambda_{4}\},
\end{equation}
where $\Lambda_{3}$ is defined in \eqref{Constant-4-1}. From now on, we assume that $-\infty<\beta_{ij}< \Lambda$ for every $(i,j)\in \mathcal{K}_{2}$.
\begin{proof}[\bf Conclusion of the proof of Theorem \ref{Theorem-1}]
We will proceed by mathematical induction on the number of sub-groups. Denote by $p$ the number of sub-groups considered (that is, the cardinality of $|\Gamma|$) and $p=1,2,\ldots,m$. We have proved in the previous subsection that the result holds true for $p=1$, i.e, for all levels $c_\Gamma$ with $|\Gamma|=1$. We suppose by induction hypothesis that the result holds true for every level $c_\Gamma$ with $|\Gamma| \leq p$, for some  $1\leq p\leq m-1$. In particular, observe that the estimates of Theorem \ref{Energy Estimation-1,2,m-1} hold for $c_\Gamma$. We want to prove Theorem \ref{Theorem-1} for $c_G$, with $|G|=p+1$. Without loss of generality, we will show it for $G=\{1,\ldots, p+1\}$. Observe that the estimates of Theorem \ref{Energy Estimation-1,2,m-1} hold for $c_G$.

From Proposition \ref{PS Sequence} we know that there exists a sequence $\{\mathbf{u}_{n}\}\subset\mathcal{N}_G$ satisfying
\begin{equation}\label{thm-1}
\lim_{n\rightarrow\infty}J_G(\mathbf{u}_{n})=c_G,\quad \lim_{n\rightarrow\infty}J_G'(\mathbf{u}_{n})=0.
\end{equation}
Since $\{u^{n}_i\}$ is uniformly bounded in $H^1_0(\Omega)$, we may assume that $u^{n}_i\rightharpoonup u_i$ weakly in $H^1_0(\Omega)$. Passing to a subsequence, we may assume that
\begin{equation*}
u_{i}^{n}\rightharpoonup u_{i}, ~weakly ~in ~L^{4}(\Omega), ~i\in I_{h}, ~h=1, \ldots, p+1.
\end{equation*}
\begin{equation*}
(u_{i}^{n})^{2}\rightharpoonup u_{i}^{2}, ~weakly ~in ~L^{2}(\Omega), ~i\in I_{h}, ~h=1, \ldots, p+1.
\end{equation*}
\begin{equation*}
u_{i}^{n}\rightarrow u_{i}, ~strongly ~in ~L^{2}(\Omega), ~i\in I_{h}, ~h=1, \ldots, p+1.
\end{equation*}
Set $v_{i}^{n}=u_{i}^{n}-u_{i}$, so that
\begin{equation}\label{thm-2-1}
v_{i}^{n}\rightharpoonup 0 \text{ in } H^{1}_{0}(\Omega).
\end{equation}
Note that $\mathbf{u}_{n}\in\mathcal{N}_G$. Then for $\forall h=1, \ldots, p+1$ we have
\begin{equation}\label{thm-3-2}
\sum_{i\in I_{h}}\int_{\Omega}|\nabla u_{i}^{n}|^{2}+\lambda_{i}|u_{i}^{n}|^{2}=
\sum_{k=1}^{p+1}\sum_{(i,j)\in I_{h}\times I_{k}}\int_{\Omega}\beta_{ij}| u_{i}^{n}|^{2}|u_{j}^{n}|^{2}.
\end{equation}
We deduce from \eqref{thm-2-1} that
\begin{equation}\label{thm-2-2}
\int_{\Omega}|\nabla u_{i}^{n}|^{2}=\int_{\Omega}|\nabla v_{i}^{n}|^{2}+|\nabla u_{i}|^{2}+o(1), \quad \int_{\Omega}|u_{i}^{n}|^{2}=\int_{\Omega}u_{i}^{2}+o(1),
\end{equation}
and by Lemma \ref{BL} we have
\begin{equation}\label{thm-2}
\int_{\Omega}| u_{i}^{n}|^{2}|u_{j}^{n}|^{2}=\int_{\Omega}| v_{i}^{n}|^{2}|v_{j}^{n}|^{2}+ \int_{\Omega}| u_{i}|^{2}|u_{j}|^{2}+o(1).
\end{equation}
It follows from \eqref{thm-3-2}-\eqref{thm-2} and $J_G'(\mathbf{u})\mathbf{u}=0$ that
\begin{equation}\label{thm-3}
\sum_{i\in I_{h}}\int_{\Omega}|\nabla v_{i}^{n}|^{2}-
\sum_{k=1}^{p+1}\sum_{(i,j)\in I_{h}\times I_{k}}\int_{\Omega}\beta_{ij}| v_{i}^{n}|^{2}|v_{j}^{n}|^{2}=o(1), ~h=1, \ldots, p+1,
\end{equation}
and
\begin{equation}\label{thm-4}
J_G(\mathbf{u}_{n})=J_G(\mathbf{u})+E_G(\mathbf{v}_{n})+o(1),
\end{equation}
where we recall that $E_G$ is defined in \eqref{eq:def_of_E}.  Passing to a subsequence, we may assume that
\begin{equation*}
\lim_{n\rightarrow\infty}\sum_{i\in I_{h}}\int_{\Omega}|\nabla v_{i}^{n}|^{2}=d_{h}, ~h=1, \ldots, p+1.
\end{equation*}
Hence, by \eqref{thm-3} we have $E_G(\mathbf{v}_{n})= \frac{1}{4}\sum_{h=1}^{p+1}d_{h}+o(1)$. Letting $n\rightarrow\infty$ in \eqref{thm-4}, we see that
\begin{equation}\label{thm-5}
0\leq J_G(\mathbf{u})\leq J_G(\mathbf{u})+\frac{1}{4}\sum_{h=1}^{p+1}d_{h}=\lim_{n\rightarrow\infty}J_G(\mathbf{u}_{n})=c_G\leq \overline{C},
\end{equation}
and so we can assume that $J_G(\mathbf{u}_{n})\leq 2\overline{C}$ for $n$ large enough. Next, we show that $\mathbf{u}\in \mathcal{N}_G$ using a contradiction argument.

{\bf Case 1}. $\mathbf{u}=(\mathbf{u}_{1},\cdots, \mathbf{u}_{h},\cdots, \mathbf{u}_{p+1})\equiv \mathbf{0}$.

Firstly, we claim that $d_{h}>0, ~h=1, \ldots, p+1$. In fact, by contradiction we assume that $d_{h_{1}}=0$ for some $h_{1}$. Then we get that $v_{i}^{n}\rightarrow 0$ strongly in $H^{1}_{0}(\Omega)$ and so $u_{i}^{n}\rightarrow 0$ strongly in $H^{1}_{0}(\Omega), \forall i\in I_{h_{1}}$. Thus
$$
\lim_{n\rightarrow\infty}\sum_{i\in I_{h_{1}}}|u_{i}^{n}|_{4}^{2}=0,
$$
which contradicts Lemma \ref{Bounded}. Therefore, $d_{h}>0, ~h=1, \ldots, p+1$.
Then we see from \eqref{Vector Sobolev Inequality}, \eqref{thm-2-2}-\eqref{thm-3} that
\begin{align}\label{thm-5-8}
\sum_{i\in I_{h}}\int_{\Omega}|\nabla v_{i}^{n}|^{2}&=\sum_{(i,j)\in I_{h}^{2}}\int_{\Omega}\beta_{ij}| v_{i}^{n}|^{2}|v_{j}^{n}|^{2}+
\sum_{k\neq h}\sum_{(i,j)\in I_{h}\times I_{k}}\left(\int_{\Omega}\beta_{ij}| u_{i}^{n}|^{2}|u_{j}^{n}|^{2}+o(1)\right)+o(1)\nonumber\\
&\leq (4l_{h})^{-1}\left(\sum_{i\in I_{h}}\int_{\Omega}|\nabla v_{i}^{n}|^{2}\right)^{2}+\frac{\Lambda}{S^{2}}\sum_{k\neq h}\sum_{(i,j)\in I_{h}\times I_{k}}
\|u_{i}^{n}\|_{i}^{2}\|u_{j}^{n}\|_{j}^{2}+o(1)\nonumber\\
&\leq (4l_{h})^{-1}\left(\sum_{i\in I_{h}}\int_{\Omega}|\nabla v_{i}^{n}|^{2}\right)^{2}+\frac{8\Lambda\overline{C}}{S^{2}}\sum_{i\in I_{h}}\int_{\Omega}|\nabla u_{i}^{n}|^{2}+o(1)\nonumber\\
&\leq (4l_{h})^{-1}\left(\sum_{i\in I_{h}}\int_{\Omega}|\nabla v_{i}^{n}|^{2}\right)^{2}+\frac{8\Lambda\overline{C}}{S^{2}}\sum_{i\in I_{h}}\int_{\Omega}|\nabla v_{i}^{n}|^{2}+o(1).
\end{align}
Hence, $d_{h}\geq 4l_{h}\left(1-\frac{8\overline{C}}{S^{2}}\Lambda\right), h=1, \ldots, p+1$. It follows from \eqref{Constant-8} and \eqref{thm-5} that
\begin{equation}\label{thm-5-3}
c_G=\frac{1}{4}\sum_{h=1}^{p+1}d_{h}\geq \sum_{h=1}^{p+1}l_{h}\left(1-\frac{8\overline{C}}{S^{2}}\Lambda\right)>\sum_{h=1}^{p+1}l_{h}-\delta,
\end{equation}
which contradicts Theorem \ref{Energy Estimates}. Hence, $\mathbf{u}\not \equiv \mathbf{0}$.

{\bf Case 2}. Only one component of $\mathbf{u}$ is not zero.

Without loss of generality, we may assume that $\mathbf{u}_{1}\neq \mathbf{0}, \mathbf{u}_{2}=\cdots=\mathbf{u}_{h}=\cdots= \mathbf{u}_{p+1}=\mathbf{0}$.  Similarly to Case 1, we can get that $d_{h}>0, h=2,\ldots, p+1$. Moreover, similarly to \eqref{thm-5-8}, we deduce from \eqref{Vector Sobolev Inequality} and \eqref{thm-3} that
$d_{h}\geq 4l_{h}\left(1-\frac{8\overline{C}}{S^{2}}\Lambda\right), h=2,\ldots,p+1$. Note that $\mathbf{u}_{1}$ is a solution of \eqref{System-h-1} with $h=1$, and so $J_G(\mathbf{u}_{1}, \mathbf{0}, \cdots,\mathbf{0})\geq c_{1}$. It follows from \eqref{Constant-8} and \eqref{thm-5} that
\begin{equation*}
c_G\geq J_G(\mathbf{u}_{1}, \mathbf{0}, \cdots, \mathbf{0})+\frac{1}{4}\sum_{h=2}^{p+1} d_{h}\geq c_{1}+\sum_{h=2}^{p+1}l_{h}\left(1-\frac{8\overline{C}}{S^{2}}\Lambda\right)>c_{1}+\sum_{h=2}^{p+1}l_{h}-\delta,
\end{equation*}
a contradiction with Theorem \ref{Energy Estimation-1,2,m-1}. Therefore, Case 2 is impossible.

{\bf Case 3}. Only $q$ components of $\mathbf{u}$ are not zero, $q=2,3,\ldots, p$.

Without loss of generality, we may assume that $\mathbf{u}_{1}\neq \mathbf{0}, \mathbf{u}_{2}\neq \mathbf{0}, \cdots, \mathbf{u}_{q}\neq \mathbf{0}, \mathbf{u}_{q+1}=\cdots= \mathbf{u}_{h}=\cdots= \mathbf{u}_{p+1}=\mathbf{0}$. Similarly to Case 1, we can get that $d_{h}>0, h=q+1,\ldots,p+1$. Thus by \eqref{Vector Sobolev Inequality} and \eqref{thm-3} we get that
\begin{align*}
\sum_{i\in I_{h}}\int_{\Omega}|\nabla v_{i}^{n}|^{2}& =
\sum_{(i,j)\in I^{2}_{h}}\int_{\Omega}\beta_{ij}| v_{i}^{n}|^{2}|v_{j}^{n}|^{2}+\sum_{k\neq h}\sum_{(i,j)\in I_{h}\times I_{k}}\left(\int_{\Omega}\beta_{ij}| u_{i}^{n}|^{2}|u_{j}^{n}|^{2}+o(1)\right)+o(1) \\
   &\leq (4l_{h})^{-1}\left(\sum_{i\in I_{h}}\int_{\Omega}|\nabla v_{i}^{n}|^{2}\right)^{2}+\frac{8\Lambda\overline{C}}{S^{2}}\sum_{i\in I_{h}}\int_{\Omega}|\nabla v_{i}^{n}|^{2}+o(1),
\end{align*}
which yields that $d_{h}\geq 4l_{h}\left(1-\frac{8\overline{C}}{S^{2}}\Lambda\right), h=q+1,\ldots,p+1$. Note that $(\mathbf{u}_{1},\mathbf{u}_{2},\cdots, \mathbf{u}_{q}, \mathbf{0}, \cdots,\mathbf{0})$ is a solution of \eqref{S-system} with the number of sub-group is $q$, and so $J_G(\mathbf{u}_{1},\mathbf{u}_{2}, \cdots, \mathbf{u}_{q}, \mathbf{0}, \cdots,\mathbf{0})\geq c_{1\ldots q}$. It follows from \eqref{Constant-8} and \eqref{thm-5} that
\begin{equation*}
c_G\geq J_G(\mathbf{u}_{1}, \cdots, \mathbf{u}_{q},\mathbf{0}, \cdots, \mathbf{0})+\frac{1}{4}\sum_{h=q+1}^{p+1} d_{h}\geq c_{1\ldots q}+\sum_{h=q+1}^{p+1}l_{h}\left(1-\frac{8\overline{C}\Lambda}{S^{2}}\right)>c_{1\ldots q}+\sum_{h=q+1}^{p+1}l_{h}-\delta,
\end{equation*}
which contradicts Theorem \ref{Energy Estimation-1,2,m-1}. Therefore, Case 3 is impossible.

Since Case 1, Case 2 and Case 3 are impossible, then we get that $\mathbf{u}_{h}\neq \mathbf{0}$ for any $h=1,2,\ldots, p+1$. Therefore, $\mathbf{u} \in \mathcal{N}_G$. It follows from \eqref{thm-5} that
\begin{equation}\label{thm-5-5}
c_G\leq J_G(\mathbf{u})\leq J_G(\mathbf{u})+\frac{1}{4}\sum_{h=1}^{p+1}d_{h}=\lim_{n\rightarrow\infty}J_G(\mathbf{u}_{n})=c_G,
\end{equation}
that is, $J_G(\mathbf{u})=c_G$. It is easy to get that
\begin{equation*}
|\mathbf{u}|\in \mathcal{N}_G ~~and ~~J_G(|\mathbf{u}|)=c_G.
\end{equation*}
It follows from Lemma \ref{Lagrange} that $|\mathbf{u}|$ is a solution to system \eqref{S-system} with the number of sub-group is $p+1$.
\end{proof}

\begin{remark}\label{5-3}
We deduce from \eqref{thm-5-5} that $d_{h}=0$ for any $h=1,2,\ldots, p+1$. Then we know that $v^{n}_i\rightarrow 0$ strongly in $H^1_0(\Omega)$, and so
 $u^{n}_i\rightarrow u_i$ strongly in $H^1_0(\Omega)$.
\end{remark}

\subsection{Proof of the remaining results}

All the remaining results follow either directly from Theorem \ref{Theorem-1} or exactly as results in other papers.

Regarding Corollary \ref{Classification}, the existence of ground states follow directly from Theorem \ref{Theorem-1} in the case $m=1$ and $\mathbf{a}=(0,d)$. Indeed, $\tilde c$ is achieved by a solution of \eqref{S-system}. Therefore $\tilde c=\inf\{J(\mathbf{u}):\  J'(\mathbf{u})= 0,\ \mathbf{u}\in H^1_0(\Omega;\R^d), \mathbf{u}\neq \mathbf{0} \}$, the ground state level. The classification result, on the other hand, follows applying exactly as in \cite[Theorem 2.1]{Tavares 2016}  (\emph{cf.} also the proof of Theorem \ref{limit system-6-1} in Section \ref{sec2}).

Corollary \ref{Theorem-1-1} is a direct consequence of Theorem \ref{Theorem-1} in the case $m=d$, $\mathbf{a}=(0,1,\ldots, d)$.

Finally, having established that $c$ is achieved in Theorem \ref{Theorem-1}, the proofs of Theorems \ref{Theorem-2} and \ref{Theorem-3} follow word by word the ones of Theorem 1.4 and Theorem 1.5 in \cite{Tavares 2016-1}.

\medbreak

\noindent{\bf Acknowledgments}

H. Tavares is partially supported by the Portuguese government through FCT/Portugal - Funda\c c\~ao para a Ci\^encia e a Tecnologia, I.P. under the project PTDC/MAT-PUR/28686/2017 and through the grant UID/MAT/04561/2013.

S. You would like to thank the China Scholarship Council of China (NO.201806180084)
for financial support during the period of his overseas study and to express his gratitude to
the Department of Mathematics, Faculty of Sciences of the University of Lisbon for its kind hospitality.

\bibliographystyle{plain}
\end{document}